\documentclass[a4paper,10pt]{article}
\usepackage{amsmath,amssymb,amsfonts,amsthm}
\usepackage{graphicx,color,subfigure,enumerate,multirow,accents}
\usepackage[]{hyperref}
\usepackage{a4wide}
\usepackage[dvipsnames]{xcolor}
\usepackage[a4paper,width={15cm},left=3cm,bottom=3.5cm, top=3.5cm]{geometry}

\newcommand\RR{\ensuremath{\mathbb{R}}}
\newcommand\CC{\ensuremath{\mathbb{C}}}

\newcommand\NN{\ensuremath{\mathbb{N}}}
\newcommand\eps{\ensuremath{\varepsilon}}

\newcommand\pot{\ensuremath{A_{a^-,a^+}}}

\newcommand\dpp{\ensuremath{\ddot{\mathcal{P}}}}

\newcommand\axis{\ensuremath{\mathcal{R}}}

\newcommand\ic{\ensuremath{\mathcal{I}_c}}
\newcommand\Rm{\ensuremath{\mathcal R_-}}
\newcommand\dpOm{\ensuremath{\ddot{\Omega}}}
\newcommand\HA{\ensuremath{H_A}}

\newtheorem{thm}{Theorem}[section]
\newtheorem{prop}[thm]{Proposition}
\newtheorem{cor}[thm]{Corollary}
\newtheorem{lem}[thm]{Lemma}
\newtheorem{dfn}[thm]{Definition}
\newtheorem{rem}[thm]{Remark}

\makeatletter
\let\@fnsymbol\@alph
\makeatother

\begin{document}

 \title{Spectral stability under removal of small capacity sets and
  applications to Aharonov-Bohm operators}
\author{L. Abatangelo\footnote{
Dipartimento di Matematica e Applicazioni,
 Universit\`a di Milano--Bicocca,
Via Cozzi 55, 20125 Milano, Italy,
\texttt{laura.abatangelo@unimib.it}}, V. Felli\footnote{
Dipartimento di Scienza dei Materiali,
 Universit\`a di Milano--Bicocca,
Via Cozzi 55, 20125 Milano, Italy,
\texttt{veronica.felli@unimib.it}}, L. Hillairet\footnote{
MAPMO (UMR6628), Universit\'e d'Orl\'eans, 45067 Orl\'eans, France,
\texttt{luc.hillairet@univ-orleans.fr}},
C. L\'ena\footnote{Dipartimento di Matematica \emph{Giuseppe Peano},
  Universit\`a di Torino, Via Carlo Alberto 10, 10123 Torino, Italy, 
\texttt{clena@unito.it}}}

\maketitle

\begin{abstract}
 We first  establish a sharp relation between the order of vanishing of a Dirichlet
eigenfunction at a point and the leading term of the  asymptotic expansion of the
Dirichlet eigenvalue variation, as a removed compact set concentrates
at that point. Then we apply this spectral stability result to the
study of the asymptotic behaviour of eigenvalues of  Aharonov-Bohm
operators  with two colliding  poles moving  on an axis of symmetry  of the
domain. 
\end{abstract}



\section{Introduction and main results}

The present paper is concerned with asymptotic estimates of the
eigenvalue variation under either removal of small sets from the
domain or operator variations due to moving  poles of singular
coefficients. More precisely, in the first part of the paper  we will
investigate the relation between the order of vanishing of a Dirichlet
eigenfunction at a point and the leading term of the  asymptotic expansion of the
Dirichlet eigenvalue variation, as a removed compact set concentrates
at that point. In the second part of the paper we will consider  
Aharonov-Bohm operators  with two poles lying on the symmetry axis of an
axially-symmetric domain
and study the asymptotic behaviour of eigenvalues as the poles move
coalescing into a fixed point. A spectral equivalence between this
class of Aharonov-Bohm operators  and the Dirichlet Laplacian will be established,
once the poles' joining segment has been removed. Thus sharp
expansions for the Aharonov-Bohm operators will be derived from those obtained in the
first part of the paper. 

\subsection{Eigenvalue variation estimates under removal of small
  capacity sets}

It is well-known that the spectrum of the
Dirichlet Laplacian on a bounded domain $\Omega\subset\RR^n$  does not change 
when a zero capacity compact set is removed from $\Omega$, see e.g. \cite{RT}. 
In the first part of the present paper we are interested in spectral stability of the Dirichlet
Laplacian and estimates of the eigenvalue variations when the domain
is perturbed by removing sets of small capacity:  we mean the
possibility that, if $K\subset \Omega$ is a compact set,  
 the $N$-th Dirichlet eigenvalue
   $\lambda_N(\Omega\setminus K)$ in $\Omega\setminus K$ may be close to $\lambda_N(\Omega)$
if (and only if) the capacity of $K$ in $\Omega$ is close to zero.
The seminal work \cite{RT} excited much interest and now a wide literature deals with this topic, 
showing that a perturbation theory can be developed in this situation.

We consider a bounded, connected open set $\Omega\subset\RR^n$. Let $K\subset \Omega$ be a compact set.
The (condenser) capacity of $K$ in $\Omega$
 is defined as 
\begin{equation}\label{eq:cap}
 \mbox{\rm Cap}_{\Omega}(K)= \inf\left\{\int_{\Omega}|\nabla f|^2:\ 
f\in H^1_0(\Omega)\text{ and }f-\eta_K\in  H^1_0(\Omega\setminus K)\right\},
\end{equation}
where $\eta_K$ is a fixed smooth function such that $\mathop{\rm
  supp}\eta_K\subset\Omega$ and $\eta_K\equiv 1$ in a neighborhood of~$K$. 
It is easy to prove  that the infimum \eqref{eq:cap} 
is achieved by a function $V_K \in H^1_0(\Omega)$ such that $V_K-\eta_K\in H^1_0(\Omega\setminus K)$, so that 
\begin{equation}\label{eq:cap2}
 \mbox{\rm Cap}_{\Omega}(K)=\int_{\Omega}\left|\nabla V_K\right|^2\,dx,
\end{equation}
where $V_K$ (capacitary potential) is the unique solution of the Dirichlet problem
\begin{equation}
\label{eqPDECap}
\begin{cases}
	-\Delta V_K 
= 0,& \mbox{in } \Omega\setminus K,\\
	V_K = 0,& \mbox{on }\partial \Omega,\\
	V_K = 1,& \mbox{on } K.
      \end{cases}
\end{equation}
By saying that $V_K$ solves \eqref{eqPDECap} we mean that $V_K\in
H^1_0(\Omega)$,  $V_K-\eta_K\in  H^1_0(\Omega\setminus K)$, and 
\[
\int_{\Omega\setminus K}\nabla
V_K\cdot\nabla\varphi\,dx=0\quad\text{for all }\varphi\in
H^1_0(\Omega\setminus K).
\]
In \cite{Courtois1995Holes}, Courtois proves spectral stability under
removal of small capacity sets in a very general context;
furthermore, \cite{Courtois1995Holes} shows that, 
when $K\subset\Omega$ is a compact
set with $\mbox{\rm Cap}_{\Omega}(K)$  close to zero, then the function
\begin{equation}\label{eq:diff}
 \lambda_N(\Omega\setminus K) - \lambda_N(\Omega) 
\end{equation}
is even differentiable with respect to $\mbox{\rm Cap}_{\Omega}(K)$.
More precisely, in \cite{Courtois1995Holes} the following result is  established.
\begin{thm}{\rm \cite[Theorem 1.2]{Courtois1995Holes}}\label{t:courtois1.2}
  Let $X$ be a compact Riemannian manifold.  Let
  $\lambda:=\lambda_{N}=\ldots=\lambda_{N+k-1}$ be a Dirichlet eigenvalue of
  $X$ with multiplicity $k$.  There exist a function
  $r:\,\RR^+\to \RR^+$ such that $\lim_{t\to0} r(t)=0$ and  a positive
  constant $\eps_N$, such that, for any compact set
  $A$ of $X$, if $\mbox{\rm Cap}_X(A)\le\eps_N$, then
 \begin{equation}\label{eq:7}
 |\lambda_{N+j}(X\setminus A)-\lambda_{N+j} - \mbox{\rm Cap}_X (A)\cdot
 \mu_A(u_{N+j}^2)| \le \mbox{\rm Cap}_X (A)\cdot r(\mbox{\rm Cap}_X (A))
\end{equation}
 where $\mu_A$ is a finite positive probability measure supported in
 $A$ defined as the renormalized singular part of $\Delta V_A$ and
 $\{u_{N},\ldots,u_{N+k-1}\}$ is an orthonormal basis of the
 eigenspace of $\lambda$ which diagonalises the quadratic form
 $\mu_A(u^2)$ according to the increasing order of its eigenvalues.
\end{thm}
We mention that, in the particular case of $A$ concentrating to a point (see Definition
\ref{d:concentrating}) estimate 
 \eqref{eq:7} is proved by Flucher in \cite[Theorem 6]{Flucher1995}.
Theorem \ref{t:courtois1.2} above provides a sharp
asymptotic expansion  of $\lambda_{N+j}(X\setminus A) -\lambda_{N+j}$ as $\mbox{\rm Cap}_X(A)\to 0$
if $\mu_A(u_{N+j}^2)\not\to0$, but in general it reduces just to estimate 
the difference $\lambda_{N+j}(X\setminus A) - \lambda_{N+j}$
without giving its sharp vanishing order when
$\mu_A(u_{N+j}^2)\to0$.
A sharp asymptotic expansion of the eigenvalue variation in the case
of $\mu_A(u_{N+j}^2)$ vanishing requires a more precise estimate than
\eqref{eq:7}. In this regard, scanning through the proof of Theorem
\ref{t:courtois1.2} given in \cite{Courtois1995Holes} (see also
\cite[Theorem 7]{Flucher1995}), one realizes
that when the eigenvalue $\lambda_N(\Omega)$ is simple
the significant quantity is instead the $u_N$-capacity defined below, 
$u_N$ being a ($L^2$-normalized) eigenfunction related to $\lambda_N(\Omega)$. 
Indeed, this can  better describe the expansion of eigenvalues'
variation, 
as stated in Theorem \ref{propL1intro} below.

For every $u\in H^1_0(\Omega)$,  we defined the
  $u$-capacity as
\begin{equation}\label{eq:fcap}
 \mbox{\rm Cap}_{\Omega}(K,u)= \inf\left\{\int_{\Omega}|\nabla f|^2:\ f\in
    H^1_0(\Omega)\text{ and }f-u\in H^1_0(\Omega\setminus K) \right\}.
\end{equation}
We note that when $u=1$ in a  neighborhood of $K$, then we recover the
definition \eqref{eq:cap}  of the condenser capacity. 
Definition \eqref{eq:fcap} can be extended to $H^1_{\rm loc}(\Omega)$ functions,
just defining, for any $u\in H^1_{\rm loc}(\Omega)$, $\mbox{\rm
  Cap}_{\Omega}(K,u):=\mbox{\rm Cap}_{\Omega}(K,\eta_Ku)$ being
$\eta_K$ as in \eqref{eq:cap}.

The infimum in \eqref{eq:fcap}  
is achieved by a function 
$V_{K,u}$ which is the unique solution of the Dirichlet problem
\begin{equation}
\label{eqPDECapD}
\begin{cases}
-\Delta V_{K,u} = 0, & \mbox{in } \Omega\setminus K,\\
V_{K,u} = 0, & \mbox{on }\partial \Omega,\\
V_{K,u} = u, & \mbox{on } K,
\end{cases}
\end{equation}
in such a way that
\begin{equation}\label{eq:fcap2}
\mbox{\rm Cap}_{\Omega}(K,u)=\int_{\Omega}\left|\nabla V_{K,u}\right|^2\,dx. 
\end{equation}
By saying that $V_{K,u}$ solves \eqref{eqPDECapD} we mean that $V_{K,u}\in
H^1_0(\Omega)$,  $V_{K,u}-u\in  H^1_0(\Omega\setminus K)$, and 
\begin{equation}\label{eq:4}
\int_{\Omega\setminus K}\nabla
V_{K,u}\cdot\nabla\varphi\,dx=0\quad\text{for all }\varphi\in
H^1_0(\Omega\setminus K).
\end{equation}
We refer to \cite[Section 2]{Courtois1995Holes} for description of the
properties of the $u$-capacity and to \cite[Section
2]{bertrand-colbois} for the specific case of the $u_1$-capacity (which is
also called \emph{Dirichlet capacity}). For our purposes, it is
important to observe the continuity properties of  the $f$-capacity
for family of concentrating compact sets
described in the remark below.
\begin{dfn}\label{d:concentrating}
 Let $\{K_\eps\}_{\eps>0}$ be a family of compact sets contained in $\Omega$.
 We say that $K_\eps$ \emph{is concentrating to a compact set $K\subset\Omega$}
 if for every open set $U\subseteq \Omega$ such that $U\supset K$ there exists $\eps_U>0$ such that 
 $U\supset K_\eps$ for every $\eps<\eps_U$.
\end{dfn}

\begin{rem}\label{r:capcon}
  Let  $\{K_\eps\}_{\eps>0}$ be a family of
  compact sets contained in $\Omega$ concentrating to a compact set
  $K\subset\Omega$ such that one of the two following conditions hold:
  \begin{enumerate}
  \item[\rm (i)]  $\mbox{\rm Cap}_{\Omega}(K)=0$;
  \item[\rm (ii)]  $K=\bigcap_{\eps>0} K_\eps$ where $K_\eps$ is 
    decreasing as $\eps\to0$ (i.e. $K_{\eps_1}\subseteq K_{\eps_2}$ if
    $\eps_1>\eps_2$).
  \end{enumerate}
Then, for all
  $f\in H^1_0(\Omega)$, 
$V_{K_\eps,f}\to V_{K,f}$ strongly in $H^1_0(\Omega)$ and
$\lim_{\eps\to 0^+}\mbox{\rm Cap}_{\Omega}(K_\eps,f)=\mbox{\rm Cap}_{\Omega}(K,f)$; in particular, 
$V_{K_\eps}\to V_{K}$ in $H^1_0(\Omega)$ and
$\lim_{\eps\to
  0^+}\mbox{\rm Cap}_{\Omega}(K_\eps)=\mbox{\rm Cap}_{\Omega}(K)$. 

The proof in the case of assumption (ii) can be found in
\cite[Proposition 2.4]{Courtois1995Holes}; for case (i) we refer to Proposition
\ref{p:conv_cap} in the appendix.
\end{rem}

The following result is essentially contained in the intermediate steps
which are developed in \cite{Courtois1995Holes} to prove estimate
\eqref{eq:7}. It provides a sharp asymptotic expansion of
\eqref{eq:diff} in terms of the $u_N$-capacity when the
eigenvalue $\lambda_N(\Omega)$ is simple. 
 We observe that the derivation of \eqref{eq:7} for 
non simple eigenvalues requires an estimate of the remaining term in
the asymptotic expansion uniformly with respect to all eigenfunctions:
 this is 
performed in \cite{Courtois1995Holes} in terms of the condenser capacity. On the other hand,
for a simple eigenvalue, the intermediate estimates  obtained in
\cite[formulas (31) and (50)]{Courtois1995Holes} in terms of the
$u_N$-capacity are enough to obtain the following sharp asymptotic expansion. 
\begin{thm}\label{propL1intro}
  Let $\lambda_N(\Omega)$ be a simple eigenvalue of the Dirichlet
  Laplacian  in a bounded, connected, and open set
    $\Omega\subset\RR^n$ and let $u_N$ be a $L^2(\Omega)$-normalized
  eigenfunction associated to $\lambda_N(\Omega)$.  Let
  $(K_{\eps})_{\eps>0}$ be a family of compact sets contained in
  $\Omega$ concentrating to a compact set $K$ with
  $\mbox{\rm Cap}_{\Omega}(K)=0$.  Then
\[
\lambda_{N}(\Omega\setminus K_\eps)=\lambda_{
  N}(\Omega)+\mbox{\rm Cap}_{\Omega}(K_\eps,u_N)+o\left(\mbox{\rm Cap}_{\Omega}(K_\eps,u_N)\right),
\quad \text{as }\eps\to 0.
\]
\end{thm}
As already mentioned, the proof of Theorem \ref{propL1intro} is contained in the proof of
\cite[Theorem 1.2]{Courtois1995Holes},
 which is based on a  method of approximation of small eigenvalues
introduced in \cite{ColinDeVerdiere1986Multiplicite} (see also
\cite[Proposition 3.1]{Courtois1995Holes}). 
 Nevertheless, for the sake of clarity and
completeness,  we present an alternative  proof in the appendix, which
relies on the use of the spectral theorem to estimate the eigenvalue variation.

As observed in \cite[Proposition 2.8]{Courtois1995Holes}, for every
eigenfunction $u$ of the Dirichlet Laplacian in $\Omega$, we have that
$\mbox{\rm Cap}_{\Omega}(K_{\eps},u) = O({\rm Cap}_{\Omega}(K_{\eps}))$  as
$\eps\to 0$. 
This in particular means that Theorem \ref{propL1intro} is sharper than Theorem \ref{t:courtois1.2}
since even the remaining term is estimated in terms of 
the $u_N$-capacity.

We mention that estimates from above and below (but not
  sharp asymptotic expansions) of the
  eigenvalue variation in terms of the $u_1$-capacity were obtained in
\cite{bertrand-colbois}, in the case of a compact Riemannian manifold
with boundary with a small subset removed.

Motivated by Theorems \ref{t:courtois1.2} and \ref{propL1intro}, we
devote the first part of the present paper to the derivation of
the  asymptotics
of the key quantity $\mbox{\rm Cap}_{\Omega}(K_\eps,u_N)$
(which tends continuously to $0$ as $K_{\eps}$
  concentrates to a compact zero-capacity set, as observed in remark
  \ref{r:capcon})
with the goal of writing  the sharp asymptotic expansion of
\eqref{eq:diff} in some relevant examples.  In particular we address the case of compact sets
concentrating to a point, which has indeed zero
capacity in any dimension greater  than or equal to 2.
We will show that the asymptotics of $\mbox{\rm Cap}_{\Omega}(K_\eps,u_N)$ 
depends on the limit point, more precisely on the order of vanishing of
$u_N$ at that point.

As a first remark in this direction, if the eigenfunction $u_N$ does
not vanish at the limit point, then the $u_N$-capacity is in fact
asymptotic to the condenser capacity (up to a constant).

\begin{prop}\label{lemCapDintro} 
  Let $\Omega\subset\RR^n$ be a bounded connected open set 
  with $n\geq 2$, let $u\in H^1_0(\Omega)\cap C^2(\Omega)$ and 
  $(K_{\eps})_{\eps>0}$ be a family of compact sets contained in
  $\Omega$ concentrating to a point $x_0\in \Omega$ such that
  $u(x_0)\neq 0$.  Then
  \begin{equation}\label{eq:NCap}
    \mbox{\rm Cap}_{\Omega}(K_{\eps},u)=u^2(x_0){\rm Cap}_{\Omega}(K_{\eps})+o({\rm Cap}_{\Omega}(K_{\eps})),
    \quad \text{ as }\eps\to 0.
  \end{equation}
\end{prop}

In view of Proposition \ref{lemCapDintro}, if the eigenfunction $u$
does not vanish at $x_0$, then the
$u$-capacity is asymptotic to the condenser capacity and the problem
of sharp asymptotics of the eigenvalue variation \eqref{eq:diff} for
$K$ concentrating at $x_0$ is reduced to the study of the behaviour of ${\rm Cap}_{\Omega}(K)$.
In dimension $2$ we succeed in proving the following sharp asymptotic expansion
of the condenser capacity of generic compact connected sets  concentrating to
a point in terms of their diameter.

\begin{prop}\label{lemCapSmall}
  Let $\Omega$ be a bounded connected open set $\Omega\subset\RR^2$.
  Let $(K_{\eps})_{\eps>0}$ be a family of compact connected sets
  contained in $\Omega$ concentrating to a point $x_0\in \Omega$.  Let
  $\delta_\eps=\mathop{\rm diam}K_\eps$, so that $\delta_\eps\to0^+$
  as $\eps\to0^+$.  Then
\[
\mbox{\rm Cap}_{\Omega}(K_{\eps})=\frac{2\pi}{|\log(\delta_\eps)|}+O\left(\frac{1}{\log^2(\delta_\eps)}\right),\quad\text{as
}\eps\to0^+.
\]
\end{prop} 
The proof of
Proposition \ref{lemCapSmall} is based on Steiner
symmetrization methods together with elliptic coordinates, see
section \ref{sec:capac-dimens-2}.

As a consequence of Theorem \ref{propL1intro}, Propositions
\ref{lemCapDintro} and \ref{lemCapSmall}, we deduce the following
sharp asymptotic expansion of the eigenvalue variation 
 \eqref{eq:diff} 
as the removed connected compact set $K$ concentrates to a point in
dimension $n=2$.
\begin{thm}\label{t:one}
  Let $\lambda_N(\Omega)$ be a simple eigenvalue of the Dirichlet
  Laplacian in  a bounded, connected, open set
    $\Omega\subset\RR^2$
 with the $L^2(\Omega)$-normalized associated
  eigenfunction $u_N$.  Let $(K_{\eps})_{\eps>0}$ be a family of
  compact connected sets contained in $\Omega$ concentrating to a
  point $x_0\in\Omega$ such that $u_N(x_0)\ne0$. Then
\[
\lambda_{N}(\Omega\setminus K_\eps)-\lambda_{ N}(\Omega)= u_N^2(x_0)
\dfrac{2\pi}{|\log\delta_\eps|} +
o\bigg(\dfrac{1}{|\log\delta_\eps|}\bigg), \quad \text{as }\eps\to 0.
\]
\end{thm}
It is worthwhile mentioning that there is a rich literature dealing
with the asymptotic expansion of the eigenvalues when small sets are
removed from the domain, in particular when the removed set is a
tubular neighborhood of a submanifold. Theorem
\ref{t:one} above has  the following counterpart in \cite[Theorem
1.4]{Courtois1995Holes}, which provides the asymptotic expansion for
$\lambda_{N}(\Omega\setminus K_\eps)-\lambda_{ N}(\Omega)$ when
$K_\eps$ is a tubular neighborhood of a closed submanifold $Y$ of
codimension $p\ge2$.
\begin{thm}{\rm \cite[Theorem 1.4]{Courtois1995Holes}}\label{t:courtois1.4}
  Let $\lambda:=\lambda_{N}=\ldots=\lambda_{N+k-1}$ be an eigenvalue
  of $X$ with multiplicity $k$.  Let
  $\{u_{N},\ldots,u_{N+k-1}\}$ be an orthonormal basis of the
  eigenspace of $\lambda$ which diagonalises the quadratic form
  $\int_Y u^2$ according to the increasing order of its eigenvalues.
  Then, if $K_\eps$ is a tubular neighborhood of a closed submanifold
  $Y$ of codimension $p\ge2$, we have for $j=0,1,\ldots,k-1$
\[
 \lambda_{N+j}(X\setminus K_\eps)-\lambda_{N+j}=\phi_p(\eps)\int_Y u^2_{N+j} + o(\phi_p(\eps))
\]
where $\phi_p(\eps)=\frac{2\pi}{|\log\eps|}$ if $p=2$ and
$\phi_p(\eps)=(p-2)(\mathop{\rm Vol } (Y))^{p-1}\eps^{p-2}$ if $p\ge3$.
\end{thm}
Theorem \ref{t:courtois1.4} generalizes preexisting results obtained
for simple eigenvalues by Ozawa \cite{Ozawa1981Duke} when $K$ is a
point and $\Omega$ is a smooth bounded domain in $\RR^2$ and by Chavel
and Feldman for any codimension $p$ \cite{chavel-feldman}.  Concerning
the case in which $K$ is a point, it is worthwhile citing also
\cite{Besson1985}, which provides the whole asymptotic expansion for
\eqref{eq:diff}.
 We highlight that, in the case $n=2$ and for simple limit
 eigenvalues, Theorem \ref{t:one} holds for general families of
 compact sets concentrating at a point, which are not required to have
 necessarily the special form of  decreasing
 neighborhoods of the limit point. The validity of the asymptotic
 expansion for general families of removed compact sets 
 finds applications in the analysis of spectral stability for magnetic
 Aharonov--Bohm operators with two coalescing points; this case
 requires the possibility of choosing as 
 $K_\eps$ a nodal line of a magnetic  eigenfunction joining the
 poles, see section \ref{sec:spectr-equiv-lapl} and \cite{AFL-2}.

 When the limit eigenfunction $u_N$ vanishes on the limit compact set,
 both Theorems \ref{t:courtois1.4} and \ref{t:one} reduces to be just
 an estimate of the vanishing rate of the eigenvalue variation,
 without giving any sharp information on the leading term of the
 expansion. Nevertheless, in view of Theorem \ref{propL1intro}, a
 sharp asymptotics for simple eigenvalues can be obtained once the
 asymptotics of ${\rm Cap}_\Omega(K_\eps,u_N)$ is computed, as we will
 do at least for special shapes of concentrating compact sets
 (i.e. segments and disks) in dimension $2$.

  Let $u$ be an eigenfunction of the Dirichlet
 Laplacian in $\Omega$, with $\Omega$ being a bounded, connected open
 set in $\RR^2$  containing $0$.  It is well-known that
 $u \in C^{\infty}(\Omega)$ and there exist
 $k\in \NN\cup\{0\}$, $\beta\in\RR\setminus\{0\}$ and
 $\alpha\in [0,\pi)$ such that
\begin{equation}\label{eq:asyphi0}
  r^{-k} u(r(\cos t,\sin t)) \to 
  \beta \sin(\alpha-kt),
 \end{equation}
in $C^{1,\tau}([0,2\pi])$ as $r\to0^+$ for any
$\tau\in (0,1)$ (see e.g. \cite{FFT}).  In this case we say 
  that $u$ has a zero
    of order $k$ at $0$.
We note that 
 \eqref{eq:asyphi0} implies that $u$ has exactly $k$ nodal lines 
dividing the $2\pi$-angle in equal parts; the minimal slope of
tangents to nodal lines is equal to $\frac{\alpha}{k}$.
 We also observe that, if $k=0$ in \eqref{eq:asyphi0}, then $\beta\sin\alpha=u(0)$.

The following  result provides the asymptotics of the $u$-capacity in
the case of 
segments concentrating at a point.
\begin{thm}  \label{t:capphiNintro}
  Let $s_\eps=[-\eps,\eps]\times \{0\}$. For $u$ being a
  $L^2(\Omega)$-normalized  eigenfunction of the Dirichlet Laplacian
  in  an open, bounded, connected set 
$\Omega\subset\RR^2$ containing $0$, let $k\in \NN\cup\{0\}$,  $\beta\in\RR\setminus\{0\}$, and $\alpha\in [0,\pi)$ 
be as in \eqref{eq:asyphi0}.
\begin{enumerate}[\rm (i)]
\item If $\alpha\neq0$, then
\begin{equation}\label{eq:2}
  \mbox{\rm Cap}_{\Omega}(s_{\eps},u)
 =\begin{cases}
   \frac{2\pi}{|\log\eps|}\, u^2(0)\,(1+o(1)), &\text{if }k=0,\\[4pt]
    \eps^{2k}\,\pi\,\beta^2\sin^2\alpha\, C_k  (1+o(1)), &\text{if }k\ge1,
  \end{cases}
	\end{equation}
as $\eps\to 0^+$,  $C_k$ being  a positive constant depending on $k$
(see \eqref{eq:1}).
\item If $\alpha=0$, then  $\mbox{\rm Cap}_{\Omega}\left(s_\eps,u\right)=
O\left(\eps^{2k+2}\right)$ as $\eps\to0^+$.
\end{enumerate}
\end{thm}
Combining Theorem \ref{propL1intro} with Theorem \ref{t:capphiNintro}
we obtain the following result.
\begin{thm}\label{t:asynull}
  Let $\lambda_N(\Omega)$ be a simple eigenvalue of the Dirichlet
  Laplacian in an open, bounded, connected set 
$\Omega\subset\RR^2$ containing $0$, with the $L^2(\Omega)$-normalized associated
  eigenfunction $u_N$.  Let $k\in \NN\cup\{0\}$,  $\beta\in\RR\setminus\{0\}$, and $\alpha\in [0,\pi)$ 
be as in expansion \eqref{eq:asyphi0} for $u_N$. For $\eps>0$ small,
let $s_\eps=[-\eps,\eps]\times \{0\}$. Then
\[
\lambda_{N}(\Omega\setminus s_\eps)-\lambda_{ N}(\Omega)= 
\begin{cases}
  \frac{2\pi}{|\log\eps|}\, u_N^2(0)\,(1+o(1)), &\text{if
  }k=0,\ \alpha\neq 0,\\[4pt]
    \eps^{2k}\,\pi\,\beta^2\sin^2\alpha\, C_k  (1+o(1)), &\text{if
    }k\ge1,\ \alpha\neq0,\\[4pt]
O\left(\eps^{2k+2}\right), &\text{if }\alpha=0,
\end{cases}
\]
as $\eps\to0^+$.
\end{thm}

\begin{rem}
We observe that the condition $\alpha=0$ means the segment
$s_\eps$  to be tangent to a nodal line of the limit eigenfunction
$u_N$. Hence Theorem \ref{t:asynull} provides  sharp asymptotics
of $\lambda_{N}(\Omega\setminus s_\eps)$ if the segment is transversal
to nodal lines of $u_N$, whereas it gives just an estimate on the
vanishing order of $\lambda_{N}(\Omega\setminus s_\eps)-\lambda_{
  N}(\Omega)$ when 
the segment is tangent to a nodal line. In this case we expect 
that the vanishing order will depend on 
the precision of the approximation between the
nodal line and the segment (e.g. if the nodal line is straight,
we have trivially that the $\mbox{\rm Cap}_{\Omega}\left(s_\eps,u\right)$
is zero and $\lambda_{N}(\Omega\setminus s_\eps)-\lambda_{
  N}(\Omega)=0$). 
\end{rem}

\begin{rem}
  In the case $k=1$, i.e. if $0$ is a regular point in the nodal set
  of $u_N$, we have that $\beta^2=|\nabla u_N(0)|^2$, hence the
  asymptotic expansion in 
Theorem \ref{t:asynull} has the form 
\[
\lambda_{N}(\Omega\setminus s_\eps)-\lambda_{
  N}(\Omega)=
 \eps^{2}\,\pi\,|\nabla u_N(0)|^2\sin^2\alpha\, C_k
 (1+o(1)),\quad\text{as $\eps\to0$}.
\]
\end{rem}

Another relevant example in which  ${\rm Cap}_\Omega(K_\eps,u)$ can be
sharply estimated in terms of the vanishing order of $u$ is given by
small disks   concentrating at a zero point of $u$.

\begin{thm} \label{t:capDiskphiNintro} Let
  $B_{\eps}=\overline{B}(0,\eps)=
  \{(x_1,x_2)\in\RR^2:\sqrt{x_1^2+x_2^2}\leq\eps\}$.
  For $u$ being an $L^2(\Omega)$-normalized eigenfunction of the Dirichlet
  Laplacian in an open, bounded, connected set 
$\Omega\subset\RR^2$ containing $0$, let $k\in \NN\cup \{0\}$,
  $\beta \in \RR\setminus \{0\}$ and $\alpha \in [0,\pi)$ be as in
  \eqref{eq:asyphi0}. Then
	\begin{equation}\label{eq:capDisks}
	\mbox{\rm Cap}_{\Omega}(B_{\eps},u)
	=\begin{cases}
	\frac{2\pi}{|\log\eps|}\, u^2(0)\,(1+o(1)), &\text{if }k=0,\\[4pt]
	2k\,\pi\,\eps^{2k}\,\beta^2  (1+o(1)), &\text{if }k\ge1,
	\end{cases}
	\end{equation}
	as $\eps\to0^+$.
\end{thm}

Combining Theorem \ref{propL1intro} and Theorem \ref{t:capDiskphiNintro}, we obtain the following result.

\begin{thm}\label{t:asynullDisk}
	Let $\lambda_N(\Omega)$ be a simple eigenvalue of the Dirichlet
	Laplacian in  an open, bounded, connected set 
$\Omega\subset\RR^2$ containing $0$
with the $L^2(\Omega)$-normalized associated
	eigenfunction $u_N$.  Let $k\in \NN\cup\{0\}$,  $\beta\in\RR\setminus\{0\}$, and $\alpha\in [0,\pi)$ 
	be as in expansion \eqref{eq:asyphi0} for $u_N$. Then
	\[
	\lambda_{N}(\Omega\setminus B_\eps)-\lambda_{ N}(\Omega)= 
	\begin{cases}
	\frac{2\pi}{|\log\eps|}\, u_N^2(0)\,(1+o(1)), &\text{if
	}k=0,\\[4pt]
	2k\,\pi\,\eps^{2k}\,\beta^2\, (1+o(1)), &\text{if }k\ge1,\\[4pt]
	\end{cases}
	\]
	as $\eps\to0^+$.
\end{thm}

\begin{rem}
	In the special case $k=1$, that is to say if $0$ is a regular point in the nodal set of $u_N$, Theorem \ref{t:asynullDisk} gives the asymptotic expansion
\[
\lambda_{N}(\Omega\setminus B_\eps)-\lambda_{
          N}(\Omega)=2\,\pi\,\eps^{2}\,|\nabla u_N(0)|^2(1+o(1))
        ,\quad\text{as $\eps\to0$}.
\]
\end{rem}

\subsection{Aharonov--Bohm potentials with varying poles}
The special attention devoted to planar domains in the first part of
the paper is well understood in the context of the applications given
in the second part to the problem of spectral stability for Aharonov--Bohm potentials with varying poles.
For $a = (a_1, a_2) \in \RR^2$,
the so-called Aharonov-Bohm magnetic potential with pole
$a$ and circulation $1/2$ is defined as 
\[
A_a(x) = \frac12 \left( \frac{ - (x_2 - a_2)}{(x_1 - a_1)^2 + (x_2 -
    a_2)^2} , \frac{x_1 - a_1}{(x_1 - a_1)^2 + (x_2 - a_2)^2} \right),
\quad x=(x_1,x_2) \in \RR^2 \setminus \{a\}.
\]
The set of papers \cite{abatangelo2015sharp, abatangelo2016leading,
  abatangelo2016boundary, bonnaillie2014eigenvalues,
 noris2015aharonov}
deals with the dependence on the pole $a$ of the spectrum
of Schr\"odinger operators with Aharonov-Bohm vector potentials,
i.e. of operators $(i\nabla +A_{a})^2$ acting on functions
$u:\RR^2\to\CC$ as
\[
(i\nabla +A_{a})^2 u=-\Delta u+2iA_{a}\cdot\nabla u+|A_{a}|^2 u.
\]
In particular, the aforementioned set of papers provides a complete
picture of sharp asymptotics for simple eigenvalues when the pole $a$
is moving in $\overline \Omega$.  Of course, one can consider even
potentials which are sum of different Aharonov--Bohm potentials with
poles located at different points in the domain, being the
differential Schr\"odinger operator defined analogously.  Concerning
this, in \cite{lena2015} the author proves continuity of eigenvalues
for Schr\"odinger operators with different Aharonov--Bohm potentials
even in the case of coalescing poles.  As an application of the
results proved in the first part of the  present paper,
in section \ref{sec:asympt-expans-coal} we begin to tackle the problem
of coalescing poles, looking for  sharp asymptotics for
simple eigenvalues.  In this direction, we obtain the following result
under a symmetry assumption on the domain.
\begin{thm}\label{t:ab}
Let $\sigma:\RR^2\to \RR^2$, 
$\sigma(x_1,x_2)=(x_1,-x_2)$.
 Let $\Omega$ be 
an open, bounded, and connected set in $\RR^2$
satisfying $\sigma(\Omega)=\Omega$ and $0\in\Omega$. 
Let $\lambda_N(\Omega)$ be a simple eigenvalue of the Dirichlet Laplacian on $\Omega$
and $u_N$ be a $L^2(\Omega)$-normalized eigenfunction associated
 to $\lambda_N(\Omega)$.
Let $k\in \NN\cup\{0\}$,  $\beta\in\RR\setminus\{0\}$, and $\alpha\in [0,\pi)$ 
be as in expansion \eqref{eq:asyphi0} for $u_N$ and assume that $\alpha\neq0$.

For $a>0$ small, let $a^-=(-a,0)$ and $a^+=(a,0)$ be the poles of the following Aharonov--Bohm potential
\begin{equation*}
	\pot(x):=
-A_{a^-}+A_{a^+}
= -\frac12\frac{(-x_2,x_1+a)}{(x_1+a)^2+x_2^2}+\frac12\frac{(-x_2,x_1-a)}{(x_1-a)^2+x_2^2}
\end{equation*}
and let $\lambda_N^a$ be the $N$-th eigenvalue for $(i\nabla +\pot)^2$.
Then
\begin{equation*}
   \lambda_N^a - \lambda_N(\Omega) 
 =\begin{cases}
   \frac{2\pi}{|\log a|}\, |u_N(0)|^2\,(1+o(1)), &\text{if }k=0,\\[4pt]
    a^{2k}\,\pi\,\beta^2\sin^2\alpha\, C_k  (1+o(1)), &\text{if }k\ge1,
  \end{cases}
\end{equation*}
as $a\to 0^+$, being $C_k$ a positive constant depending only on $k$
(see \eqref{eq:1}).
\end{thm}

We observe that the assumption $\alpha\neq0$ means that the poles are
moving along a line which is in fact not tangent to any nodal line of
the limit eigenfunction $u_N$.

The main idea behind the proof of Theorem \ref{t:ab} is the spectral
equivalence between the Aharonov--Bohm operator in an axially symmetric
domain and the Dirichlet Laplacian in the domain obtained by removing
either the segment
joining the poles or its complement in the axis. Such isospectrality result is established in
section \ref{sec:spectr-equiv-lapl} and extends the isospectrality result proved in \cite{BN-He-HH2009} 
for a single pole to the case of two
poles. Once the spectral equivalence is established,
Theorem \ref{t:ab} follows as an application of Theorem
\ref{t:asynull}.

 A weakening of the symmetry  assumption required in  Theorem
 \ref{t:ab} above presents some significant additional difficulties due
 to the general shape of nodal lines of eigenfunctions (i.e. they are
 not necessarily a straight segment);  this problem is
   treated in \cite{AFL-2} in the case $k=0$.

\section{$u$-capacity }

We devote this section to explicit calculations of 
$u$-capacity in several situations.

\subsection{General compact sets concentrating away from zeros}
In this subsection, we present the case of general domains which are
concentrating to a point away from zeros of the eigenfunction $u$ and
prove the asymptotic relation, stated in Proposition \ref{lemCapDintro},
between the $u$-capacity and the condenser capacity. In
order to derive such asymptotics we first state the following lemma,
which essentially rewrites \cite[formula (53)]{Courtois1995Holes} in a
form which is more convenient for our purposes.

\begin{lem}\label{l:53Cour}
Let $\Omega$ be a bounded, connected open set in $\RR^n$ and let $K$ be a
compact set in $\Omega$. 
Let  $\eta$ be any smooth function such that $\mathop{\rm
  supp}\eta\subset\Omega$ and $\eta\equiv 1$ in a neighborhood of~$K$. 
If $u\in H^1_0(\Omega)\cap C^2(\Omega)$ then 
	\begin{equation}\label{eq:9}
\mbox{\rm Cap}_{\Omega}(K,u)=\mathcal L(u,K)
-\int_\Omega   V_{K,u} V_K \Delta(\eta u)\,dx
-2\int_{\Omega}V_{K,u}\nabla V_K\cdot
    \nabla (\eta u)\,dx
\end{equation}
where 
\begin{equation}\label{eq:8}
\left(\min_{x\in K}u^2(x)\right) \mbox{\rm Cap}_\Omega(K)\leq 
\mathcal L(u,K)\leq
\left(\max_{x\in K}u^2(x)\right) \mbox{\rm Cap}_\Omega(K).
\end{equation}
\end{lem}
\begin{proof}
Let us first assume that $K$ is a regular compact set, meaning that
$K$ is the closure  an open smooth set. Then
\begin{align*}
  \int_{\Omega}\left|\nabla V_{K,u}\right|^2\,dx
  &=\int_{\Omega\setminus K}\left|\nabla V_{K,u}\right|^2\,dx+\int_{K}\left|\nabla u\right|^2\,dx\\
  &=\int_{\partial (\Omega\setminus K)}V_{K,u}\partial_{ \nu}
    V_{K,u}\,d\sigma
    +\int_{\partial K}u\partial_{\nu}u\,d\sigma-\int_{K}u\Delta
    u\,dx\\
 &=\int_{\partial (\Omega\setminus K)}V_{K}u\partial_{ \nu}
    V_{K,u}\,d\sigma
    +\int_{\partial K}u\partial_{\nu}u\,d\sigma-\int_{K}u\Delta(\eta u)\,dx.
\end{align*}
On the other hand
\begin{align*}
  &\int_{\partial (\Omega\setminus K)}V_{K}u\partial_{ \nu} V_{K,u}\,d\sigma
    =\int_{\partial (\Omega\setminus K)}V_{K}(\eta u)\partial_{ \nu} V_{K,u}\,d\sigma
    =\int_{\Omega\setminus K}\nabla\left(V_K\eta u\right)\cdot \nabla V_{K,u}\,dx\\
  &=\int_{\Omega\setminus K}\eta u\nabla V_K\cdot \nabla V_{K,u}\,dx
    +\int_{\Omega\setminus K}V_{K}\nabla (\eta u) \cdot \nabla V_{K,u}\,dx\\
  &=\int_{\partial(\Omega \setminus K)}uV_{K,u}\partial_{\nu}V_K\,d\sigma
    { -}\int_{\Omega\setminus K}V_{K,u}\nabla V_K\cdot \nabla (\eta u)\,dx
    +\int_{\partial(\Omega\setminus K)}V_KV_{K,u}\partial_{\nu}(\eta u)\,d\sigma\\
  &\qquad -\int_{\Omega\setminus K}V_KV_{K,u}\Delta (\eta
    u)\,dx-\int_{\Omega\setminus K}V_{K,u}\nabla V_K\cdot \nabla(\eta  u)\,dx\\
  &=\int_{\partial(\Omega \setminus K)}u^2\partial_{\nu}V_K\,d\sigma-\int_{\partial K}u\partial_{\nu}u\,d\sigma
    -\int_{\Omega\setminus K}V_KV_{K,u}\Delta(\eta u)\,dx
    -2\int_{\Omega\setminus K}V_{K,u}\nabla V_K\cdot \nabla (\eta u)\,dx.
\end{align*}
Hence we obtain that, if $K$ is regular, then 
\begin{equation}\label{eq:6}
	\mbox{\rm Cap}_{\Omega}(K,u)=\int_{{\partial}(\Omega\setminus K)}u^2\partial_{\nu}V_{K}\,d\sigma
-\int_{\Omega}V_KV_{K,u}\Delta(\eta u)\,dx-2\int_{\Omega}V_{K,u}\nabla V_K\cdot
                  \nabla (\eta u)\,dx.
              \end{equation}
If $K$ is a generic compact set, then there exist a decreasing family of regular
compact sets $\{K_\eps\}_{\eps>0}$ concentrating at $K$ such that 
 $K=\bigcap_{\eps>0} K_\eps$.
If  $\eta\in C^\infty_{\rm c}(\Omega)$ is any smooth function such
that $\eta\equiv 1$ in a neighborhood of~$K$, then $\eta\equiv 1$ also
in a neighborhood of $K_\eps$ for $\eps$ sufficiently small. Writing
\eqref{eq:6} for $K_\eps$ and $\eta$ and
passing to the limit, in view of Remark \ref{r:capcon} (case (ii)) we
obtain that
\[
\mbox{\rm Cap}_{\Omega}(K,u)=\mathcal L(u,K)
-\int_{\Omega}V_KV_{K,u}\Delta(\eta u)\,dx
  { -2\int_{\Omega}V_{K,u}\nabla V_K\cdot
    \nabla (\eta u)\,dx}
\]
with
$\mathcal L(u,K)=\lim_{\eps\to 0^+}\int_{{\partial}(\Omega\setminus
  K_\eps)}u^2\partial_{\nu}V_{K_\eps}\,d\sigma$.
By Hopf's Lemma, $\partial_{\nu}V_{K_\eps}$ is positive on
$\partial K_{\eps}$, being $\nu$ the exterior normal vector to
$\Omega\setminus K_\eps$.  Moreover, by integration by parts, we have
that
$\int_{\partial K_{\eps}}|\partial_{\nu}V_{K_{\eps}}|\,d\sigma=\mbox{\rm
  Cap}_{\Omega}(K_{\eps})$. Hence
\begin{multline*}
\left(\min_{K_{\eps}}u^2\right)\mbox{\rm Cap}_{\Omega}(K_{\eps})
\le
\min_{\partial K_{\eps}}u^2\int_{\partial K_{\eps}}\left|\partial_{\nu}V_{K_{\eps}}\right|\,d\sigma\\
\le
\int_{{\partial}(\Omega\setminus K_\eps)}u^2\partial_{\nu}V_{K_\eps}\,d\sigma
\le \max_{\partial K_{\eps}}u^2\int_{\partial
  K_{\eps}}\left|\partial_{\nu}V_{K_{\eps}}\right|\,d\sigma
\le \left(\max_{K_{\eps}}u^2\right)\mbox{\rm Cap}_{\Omega}(K_{\eps}).
\end{multline*}
By Remark \ref{r:capcon} (case (ii)) and continuity of $u$, passing to the limit in the
above estimate yields \eqref{eq:8}, thus completing the proof.\end{proof}
From Lemma \ref{l:53Cour} we derive  Proposition \ref{lemCapDintro}.
\begin{proof}[Proof of Proposition \ref{lemCapDintro}]
 Let  $\eta\in C^\infty_{\rm c}(\Omega)$ be a smooth function such
that $\eta\equiv 1$ in a neighborhood of $x_0$, so that \eqref{eq:9}
can be written for $K_\eps$ and $\eta$ for $\eps$ sufficiently small. The fact that $K_{\eps}$ concentrates to $x_0$ as $\eps\to 0$ and
  the continuity of $u$ 
  implies that 
\[
\lim_{\eps\to 0}\min_{ K_{\eps}}u^2=\lim_{\eps\to
  0}\max_{K_{\eps}}u^2=u^2(x_0),
\]
so that $\mathcal L(u,K_\eps)=u^2(x_0) \mbox{\rm
  Cap}_{\Omega}(K_{\eps})+o(\mbox{\rm Cap}_{\Omega}(K_{\eps}))$ as
$\eps\to 0^+$.
From Cauchy-Schwarz Inequality and Corollary \ref{corNormL2} we deduce
that  
\[\left|\int_{\Omega}V_{K_{\eps}}V_{K_{\eps},u}\Delta(\eta u)\,dx\right|\le
\|\Delta(\eta u)\|_{L^{\infty}(\Omega)}\|V_{K_{\eps}}\|_{L^2(\Omega)}
\|V_{K_{\eps},u}\|_{L^2(\Omega)}
=o\left(\mbox{\rm Cap}_{\Omega}(K_{\eps})\right),
\]
 as $\eps\to0^+$. According to Cauchy-Schwarz Inequality and Corollary \ref{corNormL2},
\[
\left|\int_{\Omega}V_{K_\eps,u_N}\nabla V_{K_{\eps}}\cdot
  \nabla (\eta u)\,dx\right| \le\left\|\nabla
  (\eta u)\right\|_{L^{\infty}(\Omega)}\left\|\nabla
  V_{K_{\eps}}\right\|_{L^2(\Omega)}\|V_{K_{\eps},u_N}\|_{L^2(\Omega)}
=o\left(\mbox{\rm Cap}_{\Omega}(K_{\eps})\right),
\]
as $\eps\to0^+$. Equation \eqref{eq:NCap} then follows from \eqref{eq:9}.
\end{proof}

\subsection{Capacities in dimension 2}\label{sec:capac-dimens-2}

In this subsection we present some explicit computations for
capacities of compact sets concentrating to a point in a planar
domain.  

To this aim, we first derive the following estimate of the $h$-capacity in terms of the vanishing order of
the function $h$ at the concentration point of  compact sets.
\begin{lem}\label{lem:remaining}
Let $\Omega\subset\RR^n$ be a bounded connected open set 
  with $n\geq 2$ and $0\in\Omega$ and  let $\{K_\eps\}_{\eps>0}$ be a family of compact
  sets contained in $\Omega$ such that, for some $C>0$ and $\eps$
  sufficiently small, 
\[
K_\eps\subset \overline{B}(0,C\eps).
\]
Let $h\in H^1(\Omega)$ be such that $h(x)=O(|x|^{k+1})$ and $|\nabla
h(x)|=O(|x|^k)$ as $|x|\to0$ for some $k\in\NN\cup\{0\}$. Then
\[
\mbox{\rm Cap}_\Omega(K_\eps,h)=O(\eps^{2k+n})\quad\text{as }\eps\to0.
\]
\end{lem}
\begin{proof}
By monotonicity of the $h$-capacity, it is enough to prove that 
 \begin{equation*}
   \label{eqRestCap}
   \mbox{\rm Cap}_{\Omega}\left(\overline B(0,C\eps),h\right)=O\left(\eps^{2k+n}\right) \mbox{ as } \eps \to 0.
 \end{equation*}
 To do this, let us fix a smooth function $\varphi:\RR^n\to\RR$ supported in
 $B(0,2)$ and equal to $1$ on $\overline B (0,1)$. Let us define
 \begin{equation*}
   \varphi_{\eps}(x):=\varphi\left(\frac{x}{C\eps}\right)
\quad\text{and}\quad
   h_{\eps}:=\varphi_{\eps}h.
 \end{equation*}
 The function $h_{\eps}$ coincides with $h$ on $\overline B(0,C\eps)$,
 so we have, by definition of the capacity,
 \begin{equation*}
   \mbox{\rm Cap}_{\Omega}\left(\overline B(0,C\eps),h\right)\le \int_{\Omega}|\nabla h_{\eps}|^2\,dx.
 \end{equation*}
 On the other hand, for any $x \in \Omega$,
 \begin{align*}
   \left|\nabla h_{\eps}(x)\right|^2&\le 2\left(\varphi_{\eps}^2 (x)\left|\nabla h(x)\right|^2+h ^2 (x)\left|\nabla \varphi_{\eps}(x)\right|^2\right)\\
  & =2\left(\varphi^2\left( \frac{x}{C\eps}\right)\left|\nabla
       h(x)\right|^2+\frac{1}{C^2\eps^2}h^2 (x)\left|\nabla
       \varphi\left(\frac{x}{C\eps}\right)\right|^2\right).
 \end{align*}
 Since
 $|\nabla h|=O(|x|^k)$ as $|x|\to0$ and $h_\eps$ is supported in $B(0,2C\eps)$,
 then
 $\left\|\nabla h_{\eps}\right\|_{L^{\infty}(\Omega)}\le \text{\rm const\,}\eps^k$. Therefore
 \begin{equation*}
   \int_{\Omega}|\nabla h_{\eps}|^2\,dx =O(\eps^{2k+n}),
 \end{equation*}
 which proves the claim.
\end{proof}

In order to derive sharp asymptotics in both cases of condenser capacities of generic compact sets
and of $u$-capacities of segments, a key tool is the following computation
of capacity of segments in ellipses.
For $L>0$ and $\eps>0$, we denote as 
\begin{equation}\label{eq:12}
\mathcal{E}_{\eps}(L)=\left\{(x_1,x_2)\in\RR^2:\,
  \frac{x_1^2}{L^2+\eps^2}+\frac{x_2^2}{L^2}<1\right\}
\end{equation}
 the  interior of the ellipse centered at $0$ with major semi-axis  of length $\sqrt{L^2+\eps^2}$ and minor
 semi-axis of length $L$.  Furthermore, for
 every $\eps>0$ we denote as 
\begin{equation}\label{eq:13}
s_\eps=[-\eps,\eps]\times \{0\}
\end{equation}
 the
 segment of length  $2\eps$ and center $0$ on the $x_1$-axis.

\begin{lem}\label{l:ellpse}
Let  $k\in\NN\cup\{0\}$ and let $P_k$ be a homogeneous polynomial of degree $k\ge0$, i.e
\begin{equation}\label{eq:14}
P_k(x_1,x_2)=\sum_{j=0}^{k}
 c_j {x_1}^{k-j}{x_2}^j
\end{equation}
 for some $c_0,c_1,\dots,c_k\in\RR$.
	Then, for every $L>0$, 
	\begin{equation*}
  \mbox{\rm Cap}_{\mathcal{E}_{\eps}(L)}(s_{\eps},P_k)
 =\begin{cases}
   \frac{2\pi}{|\log\eps|}\, c_0^2\,\left(1+O\left(\frac1{|\log\eps|}\right)\right) , &\text{if }k=0,\\[4pt]
    \eps^{2k}\,c_0^2\, \pi C_k (1+o(1)), &\text{if }k\ge1,
  \end{cases}
	\end{equation*}
as $\eps\to 0^+$, where 
\begin{equation}\label{eq:1}
C_k=\sum_{j=1}^k j |A_{j,k}|^2,\quad\text{being}\quad
A_{j,k}=\frac1\pi\int_0^{2\pi}(\cos\eta)^k\cos(j\eta)\,d\eta.
\end{equation}
\end{lem}
\begin{rem}
  We notice that, if $k\geq 1$, then  there exists at least a
$j\in\{1,2,\dots,k\}$ such that $A_{j,k}\neq0$, so that $C_k=\sum_{j=1}^k
j |A_{j,k}|^2\neq0$ if $k\geq1$.
\end{rem}
\begin{rem}
  As a particular case of Lemma \ref{l:ellpse} when $k=0$ and $c_0=1$
  (so that $P_k\equiv 1$), we obtain that the condenser capacity of the
  segment in the ellipse is given 
  \begin{equation}\label{eq:11}
     \mbox{\rm Cap}_{\mathcal{E}_{\eps}(L)}(s_{\eps})
 =\frac{2\pi}{|\log\eps|}\,\left(1+O\left(\frac1{|\log\eps|}\right)\right)  \quad\text{as }\eps\to0^+.
  \end{equation}
\end{rem}

\begin{proof}[Proof of Lemma \ref{l:ellpse}]
We define the elliptic coordinates $(\xi,\eta)$ (see for instance \cite{Sokolov2011}) by
\[
\begin{cases}
x_1=\eps\cosh(\xi)\cos(\eta),\\
x_2=\eps\sinh(\xi)\sin(\eta),
\end{cases}
\quad \xi\geq0,\ 0\leq\eta<2\pi.
\]
Let us note that, in these coordinates, the segment $s_{\eps}$ is
defined by $\xi=0$, whereas $\mathcal E_{\eps}$ is defined by
 $0\leq\xi < \xi_{\eps}$ and $\partial \mathcal
  E_{\eps}$ is described by the condition $\xi=\xi_\eps$, with
$\eps\sinh(\xi_{\eps})=L\,$, that is to say
\begin{equation}
\label{eqBoundary}
\xi_{\eps}=\mbox{argsinh}\left(\frac{L}{\eps}\right)
=\log\left(\frac{L}{\eps}+\sqrt{1+\frac{L^2}{\eps^2}}\right).
\end{equation}
A direct computation shows that the mapping
$\Phi:(\xi,\eta)\mapsto (x_1,x_2)$ has a Jacobian matrix of the form 
\[
J(\Phi)(\xi,\eta)=h(\xi,\eta)O(\xi,\eta),
\]
with $O(\xi,\eta)$ an orthonormal matrix and $h(\xi,\eta)>0$ in
$\RR^2\setminus s_\eps$ satisfying
\[
h^2 (\xi,\eta)={\eps}^2(\cosh ^2 \xi-\cos ^2 \eta).
\]
If we evaluate the homogeneous polynomial $P_k$ in the new set of coordinates on the segment
$s_\eps=\{(\xi,\eta):\ \xi=0\}$, we end up with 
$Q_k(\xi,\eta)=P_k(\Phi(0,\eta))=c_0 {\eps}^k (\cos\eta)^k$.
Let $W$ be the Dirichlet potential of $P_k$ in elliptic coordinates, that is 
\[
 \begin{cases}
  -\Delta W =0, &\text{in }(0,\xi_\eps)\times (0,2\pi),\\
  W=0, &\text{on }\xi=\xi_\eps,\\
  W=c_0\,{\eps}^k(\cos\eta)^k, &\text{on } \xi=0,\\
  W(\xi,0)=W(\xi,2\pi), &\text{for all } \xi\in(0,\xi_\eps).
 \end{cases}
\]
Let us consider the Fourier expansion of $W$ in elliptic coordinates:
\[
 \frac{1}{\eps^k}W(\xi,\eta)= \dfrac{a_0(\xi)}{2} + \sum_{j\ge1} \left(a_j(\xi)\cos(j\eta) + b_j(\xi)\sin(j\eta)\right)
\]
where
\begin{equation*}
 a_j(\xi)= \frac1\pi \int_0^{2\pi} \frac{1}{\eps^k}W(\xi,\eta) \cos(j\eta) \,d\eta,\quad
 b_j(\xi)= \frac1\pi \int_0^{2\pi} \frac{1}{\eps^k}W(\xi,\eta) \sin(j\eta) \,d\eta,
\end{equation*}
from which we see that $b_j(0)=0$ for any $j$ and 
$a_j(0)=0$ for all $j>k$.
Therefore we have
\[
 0=-\Delta_{(\xi,\eta)}W= \eps^k \dfrac{a_0''(\xi)}{2} 
 + \eps^k \sum_{j\ge1} \left((a_j''(\xi)-j^2a_j(\xi))\cos(j\eta) + (b_j''(\xi)-j^2b_j(\xi))\sin(j\eta)\right)
\]
and imposing the boundary conditions for $\xi\in(0,\xi_\eps)$ we obtain 
\begin{align}
 &a_0(\xi) = a_0(0) \left(1-\dfrac{\xi}{\xi_\eps}\right), \label{eq:a0}\\
 &a_j(\xi) = a_j(0) \left(\dfrac{e^{j\xi}}{1-e^{2j\xi_\eps}}+\dfrac{e^{-j\xi}}{1-e^{-2j\xi_\eps}}\right), \qquad \text{for }j\ge1 \label{eq:aj}\\
 &b_j(\xi)=0\qquad \text{for }j\ge0 \label{eq:bj}.
\end{align}
In this way 
\[
\dfrac1{\eps^k} W(\xi,\eta)= \dfrac{a_0(\xi)}{2} + \sum_{j=1}^k a_j(\xi)\cos(j\eta)
\]
and then by Parseval's  identity 
\begin{equation}\label{eq:parseval}
 \iint_{(0,\xi_\eps)\times(0,2\pi)} |\nabla W|^2 
 = \eps^{2k}  \frac\pi2 \int_{0}^{\xi_\eps} |a_0'(\xi)|^2 \,d\xi
 + \eps^{2k} \pi \sum_{j=1}^k \int_{0}^{\xi_\eps} (|a_j'(\xi)|^2 + j^2 |a_j(\xi)|^2)\,d\xi.
\end{equation}
Let us now compute every term of the latter expression. First we have
\begin{equation}\label{eq:inta0}
 \int_{0}^{\xi_\eps} |a_0'(\xi)|^2\,d\xi = \dfrac{1}{\xi_\eps}|a_0(0)|^2.
\end{equation}
Secondly, for $j\ge1$ we have
\begin{align}\label{eq:intaj'}
 \int_{0}^{\xi_\eps} |a_j'(\xi)|^2 \,d\xi
 \notag&= j^2 |a_j(0)|^2\int_{0}^{\xi_\eps} \left(\dfrac{e^{j\xi}}{1-e^{2j\xi_\eps}}-\dfrac{e^{-j\xi}}{1-e^{-2j\xi_\eps}}\right)^{\!\!2}\,d\xi\\
 \notag&= \frac{j}{2}|a_j(0)|^2 \left(\dfrac{-1}{1-e^{2j\xi_\eps}}+\dfrac{1}{1-e^{-2j\xi_\eps}} - 4j\dfrac{\xi_\eps}{(1-e^{2j\xi_\eps})(1-e^{-2j\xi_\eps})}\right)\\
 \notag&= \frac{j}{2}|a_j(0)|^2 \left( \dfrac{e^{-2j\xi_\eps}-e^{2j\xi_\eps} -4j\xi_\eps}{(1-e^{2j\xi_\eps})(1-e^{-2j\xi_\eps})} \right)\\
 &= \frac{j}{2}|a_j(0)|^2 (1+o(1)) \qquad \text{as }\eps\to0.
\end{align}
Finally we have
\begin{align}\label{eq:intaj}
 \int_{0}^{\xi_\eps} |a_j(\xi)|^2\,d\xi
 \notag&= |a_j(0)|^2\int_{0}^{\xi_\eps}\left(\dfrac{e^{j\xi}}{1-e^{2j\xi_\eps}}+\dfrac{e^{-j\xi}}{1-e^{-2j\xi_\eps}}\right)^{\!\!2}\,d\xi\\
 \notag&= |a_j(0)|^2 \dfrac{1}{2j} \left(\dfrac{-1}{1-e^{2j\xi_\eps}}+\dfrac{1}{1-e^{-2j\xi_\eps}} +4j\dfrac{\xi_\eps}{(1-e^{2j\xi_\eps})(1-e^{-2j\xi_\eps})}\right)\\
 &= |a_j(0)|^2 \dfrac{1}{2j}  (1+o(1)) \qquad \text{as }\eps\to0.
 \end{align}
Plugging \eqref{eq:inta0}, \eqref{eq:intaj'} and \eqref{eq:intaj} into \eqref{eq:parseval} we obtain
\begin{equation}\label{eq:dircap1}
 \iint_{(0,\xi_\eps)\times(0,2\pi)} |\nabla W|^2 
 = \eps^{2k}  \frac\pi2 \dfrac{1}{\xi_\eps}|a_0(0)|^2 
 + \eps^{2k} \pi \sum_{j=1}^kj |a_j(0)|^2  (1+o(1))  \qquad \text{as }\eps\to0.
\end{equation}
We note that for $k=0$ there holds $a_0(0)=2c_0$, whereas $a_j(0)=0$ for $j\ge1$.
Moreover, a simple calculation shows 
\[
\frac1{\xi_\eps}=\frac1{|\log\eps|}+O\left(\frac1{|\log\eps|^2}\right)
\]
as $\eps\to0^+$.
On the other hand, if $k\geq1$, then there exists at least a
$j\in\{1,2,\dots,k\}$ such that $a_j(0)\neq0$.

We then conclude that
\[
  \mbox{\rm Cap}_{\mathcal{E}_{\eps}(L)}(s_{\eps},P_k) =
\begin{cases}
   \frac{2\pi}{|\log\eps|} c_0^2\,\left(1+O\left(\frac1{|\log\eps|}\right)\right), &\text{if }k=0,\\[4pt]
    \eps^{2k}\, \pi \left(\sum_{j=1}^kj |a_j(0)|^2\right)  (1+o(1)), &\text{if }k\ge1,
  \end{cases}
\]
thus completing the proof.
\end{proof}

\subsubsection{Condenser capacity in dimension 2} 

We first consider
generic compact connected sets and prove the sharp asymptotic expansion
of the condenser capacity in terms of their diameter, as stated in Proposition \ref{lemCapSmall}.

\begin{proof}[Proof of Proposition \ref{lemCapSmall}]
Let $a_{\eps}, b_{\eps}\in K_{\eps}$ be such that $|b_{\eps}-a_{\eps}|=\delta_{\eps}$. 
We denote by $m_{\eps}$ the middle point of $a_{\eps}$ and $b_{\eps}$,
i.e. $m_{\eps}=\frac12(a_{\eps}+b_{\eps})$. 
Note that $m_{\eps}\to x_0$ as $\eps\to 0$.

Let us first derive an upper bound for
$\mbox{\rm Cap}_{\Omega}(K_{\eps})$.  There exists $R>0$ such that
$B(m_\eps,R)\subset \Omega$ and $B(x_0,R)\subset \Omega$ for $\eps$
sufficiently small.  According to the monotonicity properties of the
capacity, we have
\[
\mbox{\rm Cap}_{\Omega}(K_{\eps})\le \mbox{\rm
  Cap}_{B(m_\eps,R)}\overline B(m_\eps,\delta_\eps) =\mbox{\rm
  Cap}_{B(0,R)}\overline B(0,\delta_\eps).
\]
It is easy to compute $\mbox{\rm Cap}_{B(0,R)}\overline B(0,\delta_\eps)$. 
Indeed, the
radial function $V$ defined as $V(x)=f(|x|)$ with 
\[
f(r)=
\begin{cases}
1,&\mbox{if } r \le  \delta_\eps,\\
\frac{\log(r/R)}{\log(\delta_\eps/R)}, &\mbox{if }\delta_\eps  <r\le R,
\end{cases}
\]
belongs to $H^1_0(B(0,R))$, is harmonic in $B(0,R)\setminus
B(0,\delta_\eps)$ and equal to $1$ on $B(0,\delta_\eps)$. Hence $V$ is a capacitary potential and 
\begin{equation*}
\mbox{\rm Cap}_{B(0,R)}\overline
B(0,\delta_\eps)=\int_{B(0,R)}\left|\nabla V\right|^2\,dx=
2\pi\int_{\delta_\eps}^{R}\frac{dr}{r\log^2(\delta_\eps/R)}
		=\frac{2\pi}{\log(R/\delta_\eps)}.
              \end{equation*}
We therefore have
	\begin{equation}\label{eq:up}
          \mbox{\rm Cap}_{\Omega}(K_{\eps})\le
\frac{2\pi}{\log(R/\delta_\eps)}.
\end{equation}
To find a lower bound for $\mbox{\rm Cap}_{\Omega}(K_{\eps})$ is a
more delicate issue.  Since $\Omega$ is bounded, there exists a length
$L$ such that $\Omega \subset \widetilde{\mathcal{E}}_{\eps}$, where
$\widetilde{\mathcal{E}}_{\eps}$ is  the interior of the ellipse centered at $m_{\eps}$, whose
major semi-axis has length $\sqrt{L^2+\frac14\delta_\eps^2}$ and belongs to the straight
line $\mathcal D_\eps$ passing through $a_{\eps}$ and $b_{\eps}$, and
whose minor semi-axis has length $L$.  By monotonicity of the
capacity,
\[
\mbox{\rm Cap}_{\widetilde{\mathcal{E}}_{\eps}}(K_{\eps})\le \mbox{\rm
  Cap}_{\Omega}(K_{\eps}).
\]
We now claim that, if ${\widetilde s}_{\eps}$ denotes the segment of
        extremities $a_{\eps}$ and $b_{\eps}$,
\begin{equation}\label{eq:10}
\mbox{\rm Cap}_{\widetilde{\mathcal{E}}_{\eps}}({\widetilde s}_{\eps})\le \mbox{\rm
  Cap}_{\widetilde{\mathcal{E}}_{\eps}}(K_{\eps}).
\end{equation}
To prove claim \eqref{eq:10}, we first consider a regular connected
compact set $\widetilde K_\eps$ such that $K_\eps\subseteq \widetilde
K_\eps\subset \widetilde{\mathcal E}_\eps$. Since $\widetilde K_\eps$ is regular, we have that 
its capacitary potential $V_{\widetilde K_\eps}$ is continuous in
$\widetilde {\mathcal E}_\eps$. For every $x\in \RR^2$, let us denote as $S_x$ the straight line perpendicular to
$\mathcal D_\eps$ passing through
$x$. 
Let us consider the  Steiner symmetrization of
$V_{\widetilde K_\eps}$  with respect to the line $\mathcal D_{\eps}$
(see e.g. \cite{capriani}),
i.e.
\[
V^*_{\widetilde K_\eps}(x)=\inf\left\{t>0:\mathcal H^1(\{y\in
  S_x:V_{\widetilde K_\eps}(y)>t\})\leq 2 \mathop{\rm dist}(x,\mathcal
  D_\eps)\right\},
\]
where $\mathcal H^1$  is the $1$-dimensional Hausdorff measure.

Since $\widetilde K_\eps$ is connected  and
$a_\eps,b_\eps\in \widetilde K_\eps$, we have that $S_x\cap
\widetilde K_\eps \neq\emptyset$ for every $x\in
\widetilde s_\eps$. It follows that, for every $x\in \widetilde s_\eps$,
$\sup_{S_x\cap \widetilde {\mathcal{E}}_{\eps}}V_{\widetilde K_\eps}=1$; then $\mathcal H^1(\{y\in
  S_x:V_{\widetilde K_\eps}(y)>t\})=0$ if and only if $t\geq1$. It
  follows that $V^*_{\widetilde K_\eps}(x)=1$ for every $x\in \widetilde s_\eps$.
Then 
\[
\mbox{\rm Cap}_{\widetilde {\mathcal{E}}_{\eps}}(\widetilde s_{\eps})\le \int_{\widetilde{\mathcal
  E}_\eps}|\nabla V^*_{\widetilde K_\eps}|^2\,dx.
\]
 Since Steiner
symmetrization decreases the Dirichlet energy, we obtain also that 
\[
\int_{\widetilde {\mathcal
  E}_\eps}|\nabla V^*_{\widetilde{K_\eps}}|^2\,dx\leq \int_{\widetilde {\mathcal
  E}_\eps}|\nabla V_{\widetilde {K_\eps}}|^2\,dx=\mbox{\rm Cap}_{\widetilde {\mathcal{E}}_{\eps}}(\widetilde K_{\eps})
\]
thus concluding that
$\mbox{\rm Cap}_{\widetilde {\mathcal{E}}_{\eps}}(\widetilde s_{\eps})\le \mbox{\rm
  Cap}_{\widetilde {\mathcal{E}}_{\eps}}(\widetilde
K_{\eps})$. Finally, to obtain \eqref{eq:10} it is enough to
approximate $K_\eps$ by regular connected compact sets and invoke
Remark \ref{r:capcon} (ii).

Since a  roto-translation  transforms $\widetilde
{{\mathcal{E}}_{\eps}}$ into $\mathcal E_{\delta_\eps/2}$ and
$\widetilde s_\eps$ into $s_{\delta_\eps/2}$ (see the notations
introduced in \eqref{eq:12} and \eqref{eq:13}), from 
 \eqref{eq:11} it follows that 
\[
\mbox{\rm Cap}_{\widetilde {\mathcal{E}}_{\eps}}(\widetilde s_{\eps})=
\mbox{\rm Cap}_{\mathcal E_{\delta_\eps/2}}(s_{\delta_\eps/2})
 =\frac{2\pi}{|\log\delta_\eps|}\,\left(1+O\left(\frac1{|\log\delta_\eps|}\right)\right)  \quad\text{as }\eps\to0^+.
\]
Putting the above inequalities and computations together, we get
\begin{equation}\label{eq:down}
  \mbox{\rm Cap}_{\Omega}(K_{\eps})\ge
  \frac{2\pi}{|\log\delta_\eps|}\,\left(1+O\left(\frac1{|\log\delta_\eps|}\right)\right)   \quad\text{as }\eps\to0^+.
\end{equation}
Putting together Equations \eqref{eq:up} and \eqref{eq:down}, we obtain
\begin{equation}\label{eq:updown}
 \frac{2\pi}{|\log\delta_\eps|}\,\left(1+O\left(\frac1{|\log\delta_\eps|}\right)\right) 
  \le \mbox{\rm Cap}_{\Omega}(K_{\eps}) \le \frac{2\pi}{\log(R/\delta_\eps)}.
\end{equation}
Observing that
\[
\frac{2\pi}{\log(R/\delta_\eps)}=\frac{2\pi}{|\log
  \delta_\eps|}\left(1+O\left(\tfrac{1}{|\log
  \delta_\eps|}\right)\right)
\]
as $\delta_\eps\to 0^+$, we conclude the proof. 
\end{proof}

\subsubsection{$u$-capacity for segments concentrating to a point in dimension 2}

We now compute the $u$-capacity for the special shape of segments of length $2\eps$ centered at $0$; for
this particular shape we are able to consider even the case when the
limit point is a zero of the eigenfunction $u$. The interest in computing
the $u$-capacity of segments with coalescing extremities is motivated by
the remarkable application to eigenvalue asymptotics for Aharonov-Bohm
operators with two poles 
presented in  section \ref{sec:asympt-expans-coal}.

The following proposition gives the asymptotics for $h$-capacity of
concentrating segments in a planar domain when $h$ is a homogeneous polynomial.
\begin{prop}
	\label{p:DirCapSmall_segments}
Let $\Omega\subset\RR^2$ be a bounded connected open set 
  with $0\in\Omega$.	For $\eps>0$ small, let $s_\eps$ be as in
  \eqref{eq:13} and $P_k$ be a homogeneous polynomial of degree
  $k\ge0$ as in \eqref{eq:14} 
 for some $c_0,c_1,\dots,c_k\in\RR$.
	Then
	\begin{equation}
  \mbox{\rm Cap}_{\Omega}(s_{\eps},P_k)
 =\begin{cases}
   \frac{2\pi}{|\log\eps|}\, c_0^2\,\big(1+O\big(\frac1{|\log\eps|}\big)\big), &\text{if }k=0,\\[4pt]
    \eps^{2k}\,c_0^2\, \pi C_k (1+o(1)), &\text{if }k\ge1,
  \end{cases}
	\end{equation}
as $\eps\to 0^+$, with $C_k$ as in \eqref{eq:1}.
\end{prop}
\begin{proof}
  Since $\Omega$ is open and bounded, there exist $L_2>L_1>0$ such
  that, for $\eps$ sufficiently small,
  $s_\eps\subset \mathcal{E}_{\eps}(L_1)\subset \Omega \subset
  \mathcal{E}_{\eps}(L_2)$, with
  $\mathcal{E}_{\eps}(L_1),\mathcal{E}_{\eps}(L_2)$ being as in \eqref{eq:12}.  By monotonicity of the capacity,
\[
\mbox{\rm Cap}_{\mathcal{E}_{\eps}(L_2)}(s_{\eps},P_k)\le \mbox{\rm
  Cap}_{\Omega}(s_{\eps},P_k)\le \mbox{\rm
  Cap}_{\mathcal{E}_{\eps}(L_1)}(s_{\eps},P_k).
\]
The conclusion then follows from Lemma \ref{l:ellpse}.
\end{proof}

The following proposition gives, for every sufficiently smooth
function $u$, a sharp relation between the asymptotics of $\mbox{\rm
  Cap}_{\Omega}(s_\eps,u)$ 
and the order of vanishing of
$u$ at $0\in\Omega$. 

\begin{prop}  \label{p:n}
Let $\Omega\subset\RR^2$ be an open, bounded, connected set 
with $0\in\Omega$ and  let $k\in
\NN\cup\{0\}$. Let us assume that
$u\in C^{k+1}_{\rm loc}(\Omega)\setminus\{0\}$ has vanishing order at $0$ equal to
$k$, i.e. the Taylor polynomial of $u$ of order $k$ and center
  $0$ has degree $k$ and  is non-zero and $k$-homogeneous, namely is
  of the form 
\[
P_k(x_1,x_2)=\sum_{j=0}^{k}
 c_j {x_1}^{k-j}{x_2}^j
\]
 for some $c_0,c_1,\dots,c_k\in\RR$,
 $(c_0,c_1,\dots,c_k)\neq(0,0,\dots,0)$.  
\begin{enumerate}[\rm (i)]
\item If $c_0\neq0$, then
\begin{equation}\label{eq:2}
  \mbox{\rm Cap}_{\Omega}(s_{\eps},u)
 =
\begin{cases}
   \frac{2\pi}{|\log\eps|}\, u ^2 (0)\,\big(1+o(1)\big), &\text{if }k=0,\\[4pt]
    \eps^{2k}\,c_0^2\, \pi C_k (1+o(1)), &\text{if }k\ge1,
  \end{cases}
	\end{equation}
as $\eps\to 0^+$,  $C_k$ being  defined in \eqref{eq:1}.
\item If $c_0=0$, then  $\mbox{\rm Cap}_{\Omega}\left(s_\eps,u\right)=
O\left(\eps^{2k+2}\right)$ as $\eps\to0^+$.
\end{enumerate}
\end{prop}
\begin{proof}
From  the  Taylor formula, we can write $u$ as
$u=P_k+h$ for some $h\in  C^{k+1}_{\rm loc}(\Omega)$ satisfying 
\[
h(x)=O(|x|^{k+1})\quad\text{and}\quad
|\nabla h(x)|=O(|x|^{k})\quad\text{as }|x|\to0^+.
\]
We denote by $V$, $W_0$, and $W$ the capacitary potentials associated
with the capacities $\mbox{\rm Cap}_{\Omega}\left(s_\eps,u\right)$,
$\mbox{\rm Cap}_{\Omega}\left(s_\eps,P_k\right)$, and
$\mbox{\rm Cap}_{\Omega}\left(s_\eps,h\right)$ respectively. By
linearity of the Dirichlet problem, $V=W_0+W$.  Therefore we have that
	\begin{equation*}
		\mbox{\rm Cap}_{\Omega}\left(s_\eps,u\right)=\int_{\Omega}\left|\nabla V\right|^2\,dx
		=\int_{\Omega}\left|\nabla W_0\right|^2\,dx+2\int_{\Omega}\nabla W_0\cdot \nabla W\,dx
		+\int_{\Omega}\left|\nabla W\right|^2\,dx.
	\end{equation*}
By Lemma \ref{lem:remaining} we have that, as $\eps\to0^+$,
\begin{equation*}
  \int_{\Omega}\left|\nabla W\right|^2\,dx=O\left(\eps^{2k+2}\right)
\end{equation*}
and
\begin{equation*}
  \left|\int_{\Omega}\nabla W_0\cdot \nabla W\,dx\right|\le \left\|\nabla W_0\right\|_{L^{2}(\Omega)}\left\|\nabla W\right\|_{L^{2}(\Omega)}=\sqrt{\mbox{\rm Cap}_{\Omega}\left(s_\eps,P_k\right)}\,O\left(\eps^{k+1}\right).
\end{equation*}
Hence 
\begin{equation}\label{eq:3}
\mbox{\rm Cap}_{\Omega}\left(s_\eps,u\right)=\mbox{\rm Cap}_{\Omega}\left(s_\eps,P_k\right)
+
\sqrt{\mbox{\rm Cap}_{\Omega}\left(s_\eps,P_k\right)}\,O\left(\eps^{k+1}\right)+
O\left(\eps^{2k+2}\right),
\end{equation}
 as $\eps\to0^+$.
In view of Proposition \ref{p:DirCapSmall_segments} and \eqref{eq:3}, we have that, if
$c_0\neq0$, 
\[
\mbox{\rm Cap}_{\Omega}\left(s_\eps,u\right)=
\mbox{\rm Cap}_{\Omega}\left(s_\eps,P_k\right)(1+o(1))
\]
 as $\eps\to0^+$, from which estimate \eqref{eq:2} follows thanks to Proposition
\ref{p:DirCapSmall_segments}. 

On the other hand, if $c_0=0$, then Proposition \ref{p:DirCapSmall_segments}
implies that $\mbox{\rm Cap}_{\Omega}\left(s_\eps,P_k\right)=0$, hence
from \eqref{eq:3} it follows that $\mbox{\rm Cap}_{\Omega}\left(s_\eps,u\right)=
O\left(\eps^{2k+2}\right)$ as $\eps\to0^+$.  
\end{proof}

The proof of  Theorem \ref{t:capphiNintro}  (and consequently of
Theorem \ref{t:asynull}) now follows as a particular
case of Proposition \ref{p:n}.
\begin{proof}[Proof of Theorems \ref{t:capphiNintro} and \ref{t:asynull}]
  From the fact that $u\in C^\infty(\Omega)$ and
  \eqref{eq:asyphi0} it follows that the Taylor polynomial of the 
  function $u$ with center
  $0$ and order $k$ is harmonic,  $k$-homogeneous, and has degree $k$; more precisely
  it has the form
\[
P_k(r\cos t,r\sin t)=\beta r^k \sin(\alpha-kt).
\]
Since it can be also written as          $P_k(x_1,x_2)=\sum_{j=0}^{k}
 c_j {x_1}^{k-j}{x_2}^j$ for some $c_0,c_1,\dots,c_k\in\RR$, we have
 then that necessarily $c_0=\beta\sin\alpha$.
The proof of Theorem \ref{t:capphiNintro} then follows from
Proposition \ref{p:n}.
Finally, the proof of Theorem \ref{t:asynull} is a direct consequence
of Theorems \ref{propL1intro} and \ref{t:capphiNintro}.
\end{proof}

\subsubsection{$u$-capacity for small  disks   concentrating to a point in dimension 2}

We conclude this section with a proof of Theorem
\ref{t:capDiskphiNintro}. As in the proof of Theorem
\ref{t:capphiNintro}, we rely on explicit computation of the
$u$-capacity in a special case. The following result is the
counterpart, in the case of the disks, of Lemma \ref{l:ellpse}, which
was stated for segments.

\begin{lem}\label{l:disks}
  Let $k\in \NN$, $k\ge1$, and let $P_k$ be a homogeneous polynomial
  of degree $k$. Let us define the Fourier coefficients
\begin{equation*}
  a_{j,k}=\frac1{\pi}\int_0^{2\pi}P_k(\cos t, \sin t)\,\cos(jt)\,dt \quad\text{for}\quad j\in\{0,1\dots,k\}
\end{equation*}
and
\begin{equation*}
  b_{j,k}=\frac1{\pi}\int_0^{2\pi}P_k(\cos t, \sin t)\,\sin(jt)\,dt \quad\text{for}\quad j\in\{1\dots,k\}.
\end{equation*}
Then, for every $R>0$,
\begin{equation*}
  \mbox{\rm Cap}_{B(0,R)}(B_{\eps},P_k)=\pi\,D(P_k)\,\eps^{2k}(1+o(1))
\end{equation*}
as $\eps\to 0^+$, where $B_{\eps}=\overline B(0,\eps)$ and $D(P_k)$ is
a constant depending only on the coefficients of the polynomial $P_k$
given by
\begin{equation}\label{eq:18}
  D(P_k)=\frac{k\,a_{0,k}^2}{4}+\sum_{j=1}^k\frac{(k+j)^2}{2k}(a_{j,k}^2+b_{j,k}^2).
\end{equation}
\end{lem}	

\begin{proof}
  Let us denote by $V$ the Dirichlet potential of $P_k$ and by $W$ its
  expression in polar coordinates, that is to say
  $V(r \cos t,r \sin t)=W(r,t)$ for
  $(r,t)\in (0,R)\times (0,2\,\pi)$.  By definition of the Fourier
  coefficients,
\begin{equation}
  P_k(r \cos t,r\sin t)=r^k\frac{a_{0,k}}2+\sum_{j=1}^kr^k\,(a_{j,k}\cos(jt)+b_{j,k}\sin(jt)) 
\end{equation}
for all $(r,t)\in (0,+\infty)\times (0,2\,\pi)$.  For all $x\in B_{\eps}$,
$V(x)=P_k(x)$, and therefore, using polar coordinates,
\begin{multline*}
  \int_{B(0,\eps)}\left|\nabla
    V(x)\right|^2\,dx=\int_0^{\eps}r^{2k-2}\bigg[\int_0^{2\pi}k^2\,\bigg(\frac{a_{0,k}}2+\sum_{j=1}^ka_{j,k}\cos(jt)+b_{j,k}\sin(jt)\bigg)^{\!\!2}\\+\bigg(\sum_{j=1}^kj\,b_{j,k}\cos(jt)-j\,a_{j,k}\sin(jt)\bigg)^{\!\!2}\bigg]\,r\,dr.
\end{multline*}
	By Parseval's identity, we obtain
\begin{equation}\label{eq:EnInner}
  \int_{B(0,\eps)}\left|\nabla V(x)\right|^2\,dx=\frac{\eps^{2k}}{2k}\left(\frac{k^2\,\pi\,a_{0,k}^2}2+\sum_{j=1}^k\pi\, (k^2+j^2)\,(a_{j,k}^2+b_{j,k}^2)\right).
\end{equation}
Let us now determine $V$ in the open set $B(0,R)\setminus B_{\eps}$, that is to say $W(r,t)$ for $r\in (\eps,R)$. The function $W$ satisfy the boundary value problem
\begin{equation}\label{eq:polar}
\begin{cases}
  \frac1r \frac\partial{\partial r}\left(r\frac\partial{\partial r}W\right)+\frac1{r^2}\frac{\partial^2}{\partial t^2}W =0, &\text{in }(\eps,R)\times (0,2\,\pi),\\
  W(R,t)=0, &\text{for all }t\in(0,2\,\pi),\\
  W(\eps,t)=P_k(\eps\,\cos t,\eps\,\sin t), &\text{for all }t\in(0,2\,\pi),\\
  W(r,0)=W(r,2\pi), &\text{for all } r\in(\eps,R).
\end{cases}
\end{equation}
To solve problem \eqref{eq:polar}, we expand $W$ in Fourier series
with respect to the variable $t$:
\begin{equation*}
  W(r,t)=\frac{a_0(r)}2+\sum_{j\ge 1} a_j(r)\cos(jt)+b_j(r)\sin(jt),
\end{equation*}
for $(r,t)\in(\eps,R)\times (0,2\,\pi)$. Then 
\begin{multline*}
  \left(\frac1r \frac\partial{\partial r}\left(r\frac\partial{\partial
        r}W\right)+\frac1{r^2}\frac{\partial^2}{\partial
      t^2}W\right)(r,t)=
  \frac1{2r}\left(r\,a_0'(r)\right)'\\+\sum_{j=1}^k
  \left(\frac1r\left(r\,a_j'(r)\right)'-\frac{j^2}{r^2}a_j(r)\right)\,\cos(jt)+\left(\frac1r\left(r\,b_j'(r)\right)'-\frac{j^2}{r^2}b_j(r)\right)\,\sin(jt),
\end{multline*}
so that 
\begin{equation*}
  (r\,a_0'(r))'=0 \quad\text{in }  (\eps,R),
\end{equation*}
and, for $j\ge1$,
\begin{equation*}
  r\,(r\,a_j'(r))'-j^2a_j(r)=0 \quad\text{and}\quad r\,(r\,b_j'(r))'-j^2b_j(r)=0.
\end{equation*}
Taking into account the boundary conditions in  \eqref{eq:polar}, we find
\begin{equation*}
  a_0(r)=a_{0,k}\,\eps^k\frac{\log\left(\frac{r}R\right)}{\log\left(\frac{\eps}R\right)},
\end{equation*}
and, for $j\in\{1,\dots,k\}$,
\begin{equation*}
  a_j(r)=a_{j,k}\eps^k\frac{\left(\frac{R}{r}\right)^j-\left(\frac{r}{R}\right)^j}{\left(\frac{R}{\eps}\right)^j-\left(\frac{\eps}{R}\right)^j}\quad\text{and}\quad b_j(r)=b_{j,k}\eps^k\frac{\left(\frac{R}{r}\right)^j-\left(\frac{r}{R}\right)^j}{\left(\frac{R}{\eps}\right)^j-\left(\frac{\eps}{R}\right)^j},
\end{equation*}
while,	 for $j\ge k+1$, $a_j(r)=0$ and $b_j(r)=0$. Using polar coordinates and Parseval's identity as above, we find
\begin{multline*}
  \int_{B(0,R)\setminus B_{\eps}}\left|\nabla
    V(x)\right|^2\,dx\\=\int_{\eps}^{R}\bigg(\frac{\pi}{2}|a_0'(r)|^2+\pi\sum_{j=1}^k
\left(    |a_j'(r)|^2+\frac{j^2}{r^2}|a_j(r)|^2+|b_j'(r)|^2+\frac{j^2}{r^2}|b_j(r)|^2\right)\bigg)\,r\,dr.
\end{multline*}
We have 
\begin{equation*}
  \int_{\eps}^R\frac{\pi}{2}|a_0'(r)|^2\,r\,dr=\frac{\pi\,a_{0,k}^2\,\eps^{2k}}{2\log^2\left(\frac{R}{\eps}\right)}\int_{\eps}^R\frac{dr}r=\frac{\pi\,a_{0,k}^2\,\eps^{2k}}{2\log\left(\frac{R}{\eps}\right)}.
\end{equation*}
For $j\in\{1,\dots,k\}$, an integration by parts gives us
\begin{align*}
  \int_{\eps}^R
  \left(r\,|a_j'(r)|^2+\frac{j^2}{r}|a_j(r)|^2\right)\,dr&=\left[r\,a_j(r)\,a_j'(r)\right]_{\eps}^R-
  \int_{\eps}^R
  \left(\left(r\,a_j'(r)\right)'-\frac{j^2}{r}a_j(r)\right)\,a_j(r)\,dr\\
&=-\eps\,a_j(\eps)\,a_j'(\eps).
\end{align*}
Using the expression computed for $a_j(r)$, we find
\begin{equation*}
  \int_{\eps}^R \left(r\,|a_j'(r)|^2+\frac{j^2}{r}|a_j(r)|^2\right)\,dr=j\,a_{j,k}^2\,\eps^{2k}\frac{1+\left(\frac{\eps}{R}\right)^{2j}}{1-\left(\frac{\eps}{R}\right)^{2j}}.
\end{equation*}
The exact same computation gives us
\begin{equation*}
\int_{\eps}^R \left(r\,|b_j'(r)|^2+\frac{j^2}{r}|b_j(r)|^2\right)\,dr=j\,b_{j,k}^2\,\eps^{2k}\frac{1+\left(\frac{\eps}{R}\right)^{2j}}{1-\left(\frac{\eps}{R}\right)^{2j}},
\end{equation*}
so that
\begin{equation}
	\label{eq:EnOuter}
	\int_{B(0,R)\setminus B_{\eps}}\left|\nabla V(x)\right|^2\,dx=\frac{\pi\,a_{0,k}^2\,\eps^{2k}}{2\log\left(\frac{R}{\eps}\right)}+\pi\,\eps^{2k}\,\sum_{j=1}^kj\,(a_{j,k}^2+b_{j,k}^2)\,\frac{1+\left(\frac{\eps}{R}\right)^{2j}}{1-\left(\frac{\eps}{R}\right)^{2j}}.
\end{equation}
Combining \eqref{eq:EnInner} and \eqref{eq:EnOuter}, we get
\begin{equation*}
  \mbox{\rm Cap}_{B(0,R)}(B_{\eps},P_k)=\pi\,\eps^{2k}\,\left(\frac{k\,a_{0,k}^2}4+\sum_{j=1}^k\left(\frac{k^2+j^2}{2k}+j\right)\,(a_{j,k}^2+b_{j,k}^2)+O\left(\frac1{\left|\log(\eps)\right|}\right)\right),
\end{equation*}
and finally
\begin{equation*}
	\mbox{\rm
          Cap}_{B(0,R)}(B_{\eps},P_k)=\pi\,\left(k\frac{a_{0,k}^2}{4}+\sum_{j=1}^k\frac{(k+j)^2}{2k}(a_{j,k}^2+b_{j,k}^2)\right)\,\eps^{2k}\,(1+o(1)),
\end{equation*}
as $\eps\to0^+$.
\end{proof}
\begin{rem}
  Since the polynomial $P_k$ in Lemma \ref{l:disks} is of degree
  $k\geq1$, it is non zero, and therefore $D(P_k)>0$.
\end{rem}

We can now find the asymptotics of the $P_k$-capacity for small
balls concentrating at a point in any open set. This is the analogue
of Proposition \ref{p:DirCapSmall_segments} for segments.

\begin{prop}\label{p:DirCapSmall_disks}
Let $\Omega\subset\RR^2$ be a bounded connected open set 
	with $0\in\Omega$.	For $\eps>0$ small, let $B_{\eps}=\overline{B}(0,\eps)$ and $P_k$ be a homogeneous polynomial of degree
	$k\ge0$.
	Then
\begin{equation}\label{eq:asymDisksHarmo}
	\mbox{\rm Cap}_{\Omega}(B_{\eps},P_k)=
        \begin{cases}
\frac{2\pi c_0^2}{|\log\eps|}\big(1+O\big(\frac1{|\log\eps|}\big)\big)
,&\text{if }k=0\text{ and }P_k\equiv c_0,\\ 
\pi\,D(P_k)\,\eps^{2k}(1+o(1)),&\text{if }k\geq1,          
        \end{cases}
\end{equation}
as $\eps\to 0^+$, where $D(P_k)$ is the constant defined in
\eqref{eq:18}.
\end{prop} 

\begin{proof}
If $k=0$, i.e. if $P_k\equiv c_0$, then $\mbox{\rm
  Cap}_{\Omega}(B_{\eps},P_k)=c_0^2 \mbox{\rm Cap}_{\Omega}B_{\eps}$
and the conclusion follows from Proposition \ref{lemCapSmall}.

For $k\geq1$, let us fix two radii $0<R_1<R_2$ such that $B(0,R_1)\subset \Omega \subset B(0,R_2)$. By monotonicity of the capacity we have
\begin{equation*}
  \mbox{\rm Cap}_{B(0,R_2)}(B_{\eps},P_k)\le \mbox{\rm
    Cap}_{\Omega}(B_{\eps},P_k)\le \mbox{\rm
    Cap}_{B(0,R_1)}(s_{\eps},P_k).
\end{equation*}
	We apply Lemma \ref{l:disks} to $\mbox{\rm Cap}_{B(0,R_1)}(B_{\eps},P_k)$ and $\mbox{\rm Cap}_{B(0,R_1)}(B_{\eps},P_k)$ and obtain \eqref{eq:asymDisksHarmo}.
\end{proof}

The following proposition is the analogue for disks of Proposition
\ref{p:n} for segments.

\begin{prop}  \label{p:n_disks}
Let $\Omega\subset\RR^2$ be an open, bounded, connected set 
with $0\in\Omega$ and  let $k\in
\NN\cup\{0\}$. Let us assume that
$u\in C^{k+1}_{\rm loc}(\Omega) \setminus\{0\}$ has vanishing order at $0$ equal to
$k$, i.e. the Taylor polynomial of $u$ of order $k$ and center
  $0$ has degree $k$ and  is non-zero and $k$-homogeneous.  Then
\begin{equation*}
  \mbox{\rm Cap}_{\Omega}(B_{\eps},u)
 =
\begin{cases}
   \frac{2\pi}{|\log\eps|}\, u ^2 (0)\,\big(1+o(1)\big), &\text{if }k=0,\\[4pt]
   \pi\,D(P_k)\,\eps^{2k}(1+o(1)), &\text{if }k\ge1,
  \end{cases}
	\end{equation*}
as $\eps\to 0^+$,  $D(P_k)$ being  defined in \eqref{eq:18}.

\end{prop}
\begin{proof}
The proof follows by repeating the same arguments as in Proposition
\ref{p:n} and using Proposition \ref{p:DirCapSmall_disks} instead of
Proposition \ref{p:DirCapSmall_segments}.
\end{proof}

\begin{proof}[Proof of Theorems \ref{t:capDiskphiNintro} and \ref{t:asynullDisk}]
Arguing as in the proof of Theorem
  \ref{t:capphiNintro},  
  from the fact that $u\in C^\infty(\Omega)$ and
  \eqref{eq:asyphi0} we deduce that the Taylor polynomial of the 
  function $u$ with center
  $0$ and order $k$ is harmonic,  $k$-homogeneous, and has degree $k$; more precisely
  it has the form \[
P_k(r\cos t,r\sin t)=\beta r^k
  \sin(\alpha-kt).
\]
 Then, for $k\geq1$, 
 the Fourier coefficients $a_{j,k}$ and $b_{j,k}$ appearing in Lemma
 \ref{l:disks} 
are zero for $j\neq k$ and 
\[
a_{k,k}=\begin{cases}
  2\beta\sin\alpha,&\text{if }k=0,\\
\beta\sin\alpha,&\text{if }k\geq1,
\end{cases}\quad\text{and}\quad 
b_{k,k}=-\beta\cos\alpha.
\]
From \eqref{eq:18} it follows that, for $k\geq1$,
$D(P_k)=2k\beta^2$. Then the asymptotics stated in \eqref{eq:capDisks}
follows form Proposition \ref{p:n_disks}. 

The proof of Theorem \ref{t:asynullDisk} follows directly from
Theorems \ref{propL1intro} and \ref{t:capDiskphiNintro}.
\end{proof}

\section{Asymptotic expansion for coalescing poles of Aharonov-Bohm
  operators}\label{sec:asympt-expans-coal}

In this section we study Aharonov-Bohm operators on domains having one axis of symmetry. 
More specifically, let us define the reflection $\sigma:\RR^2\to \RR^2$ by 
\begin{equation*}\sigma(x_1,x_2)=(x_1,-x_2),\end{equation*} and let us consider $\Omega$, 
an open, bounded, and connected set in $\RR^2$ satisfying
$\sigma(\Omega)=\Omega$.  Let us consider a Schr\"odinger operator
with a purely magnetic potential of Aharonov-Bohm type, with two poles
on the axis of symmetry
\begin{equation*}
  \mathcal \axis:=\{(x_1,x_2)\in\RR^2\,:\,x_2=0\},
\end{equation*} 
each with a half-integer flux.

More precisely, let us fix two points $a^-=(a^-_1,0)$ and
$a^+=(a^+_1,0)$ in $\axis$, with $a^-_1<a^+_1$. We consider the vector
field $\pot$ defined on the doubly punctured plane
$\dpp(a^-,a^+):=\RR^2\setminus\{a^-,a^+\}$ by
\begin{equation*}
  \pot(x):= 
  -\frac12\frac{1}{(x_1-a^-_1)^2+x_2^2}(-x_2,x_1-a^-_1)+\frac12\frac{1}{(x_1-a^+_1)^2+x_2^2}(-x_2,x_1-a^+_1).
\end{equation*}
Let us note that, if we write, for any $x=(x_1,x_2)\in \dpp(a^-,a^+)$, 
\begin{equation*}
A_{a^-,a^+}(x_1,x_2)=(A_1(x_1,x_2),A_2(x_1,x_2)),
\end{equation*}
we have
\begin{equation*}
	A_{a^-,a^+}(x_1,-x_2)=(-A_1(x_1,x_2),A_2(x_1,x_2)).
\end{equation*}
Equivalently, we have, for any $x\in \dpp(a^-,a^+)$,
\begin{equation*}
	A_{a^-,a^+}(\sigma(x))=-A_{a^-,a^+}(x)S
\end{equation*}
where $S$ is a $2\times2$ symmetry matrix:
\begin{equation*}
	S:=\left(\begin{array}{cc}
		1&0\\
		0&-1\\
	\end{array}
	\right).
\end{equation*}
We work in the complex Hilbert space $L^2(\Omega)$ of complex valued square integrable functions on $\Omega$, with the scalar product defined by
\begin{equation*}
	\langle u,v\rangle :=\int_{\Omega}u\overline{v}\,dx
\end{equation*}
for $u$ and $v$ in $L^2(\Omega)$.  Our operator is the Friedrichs
extension of the differential
operator \begin{equation*}(i\nabla+\pot)^2,\end{equation*} acting on
$C^{\infty}_c(\dpOm(a^-,a^+))$, the space of smooth functions with compact
support in the doubly punctured domain
$\dpOm(a^-,a^+):=\Omega \setminus \{a^-,a^+\}$. We denote it by
$H_{A_{a^-,a^+}}$. It is a positive and self-adjoint operator, with
compact resolvent. Its spectrum therefore consists in a sequence of
real positive eigenvalues tending to $+\infty$, which we denote by
$(\lambda_k(a^-,a^+))_{k\ge 1}$.

\subsection{Gauge transformations}

We now construct suitable gauge transformations, in order to remove the magnetic potential.  
We use the notation
\begin{equation*}
	\ic(a^-,a^+):=[a^-_1,a^+_1]\times \{0\}
\end{equation*}
to denote the closed segment joining the two poles.
\begin{lem}\label{lemAngle}
  There exists a unique $C^{\infty}$-function $\varphi_{a^-,a^+}$
  defined on $\RR^2\setminus\ic(a^-,a^+)$ such that
	\begin{equation*}
	\nabla \varphi_{a^-,a^+}=\pot\ \mbox{ on } \RR^2\setminus\ic(a^-,a^+)
	\end{equation*}	
and
	\begin{equation*}
	\varphi_{a^-,a^+}(x_1,0)=0 \mbox{ for all }  x_1\in (a^+_1,+\infty).
	\end{equation*}
	Furthermore, $\varphi_{a^-,a^+}$ satisfies $\varphi_{a^-,a^+}\circ \sigma =-\varphi_{a^-,a^+}$.
\end{lem}
\begin{proof} First, we define $\theta_0:\RR^2\setminus\Rm\to \RR$, where $\Rm=(-\infty,0]\times\{0\}$, by 
\begin{equation*}
	\theta_0(x_1,x_2):= 2\arctan\left(\frac{x_2}{x_1+\sqrt{x_1^2+x_2^2}}\right).
\end{equation*}
As seen on Figure \ref{figPolarAngle}, $\theta_0(x)$ is the angular coordinate of the point $x$.
\begin{figure}
\begin{center}
\includegraphics[width=7cm]{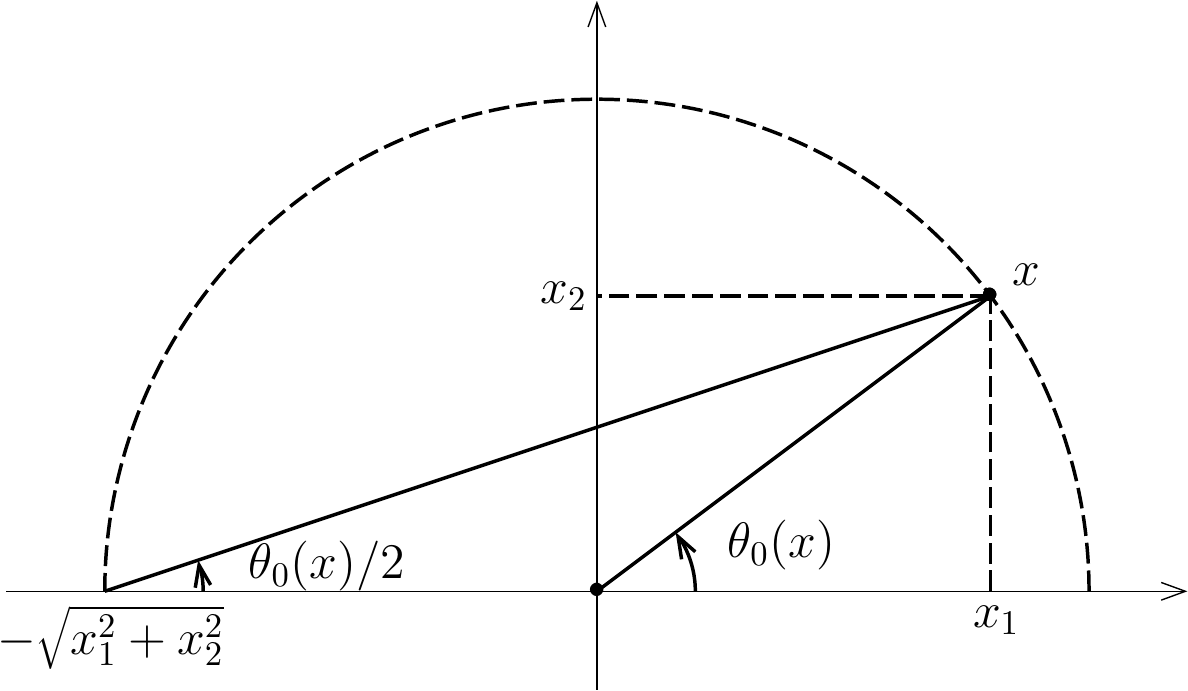}
\caption{Definition of $\theta_0(x)$.\label{figPolarAngle}}
\end{center}
\end{figure}
 The function $\theta_0$ is smooth on $\RR^2\setminus\Rm$,
 and
\begin{equation}\label{eqGradTheta}	
\nabla\theta_0(x_1,x_2)=\frac1{x_1^2+x_2^2}(-x_2,x_1)
\end{equation} 
for all $x=(x_1,x_2)\in \RR^2\setminus\Rm$. Let us note also that $\theta_0\circ\sigma =-\theta_0$.

Then we define, for any $x=(x_1,x_2)\in \RR^2\setminus ((-\infty,a^+_1]\times\{0\})$,
\begin{equation*}
  \varphi_{a^-,a^+}(x):=\frac12\theta_0(x_1-a^+_1,x_2)-\frac12\theta_0(x_1-a^-_1,x_2).
\end{equation*}
See Figure \ref{figPotential} for a geometric interpretation.
\begin{figure}
\begin{center}
\includegraphics[width=7cm]{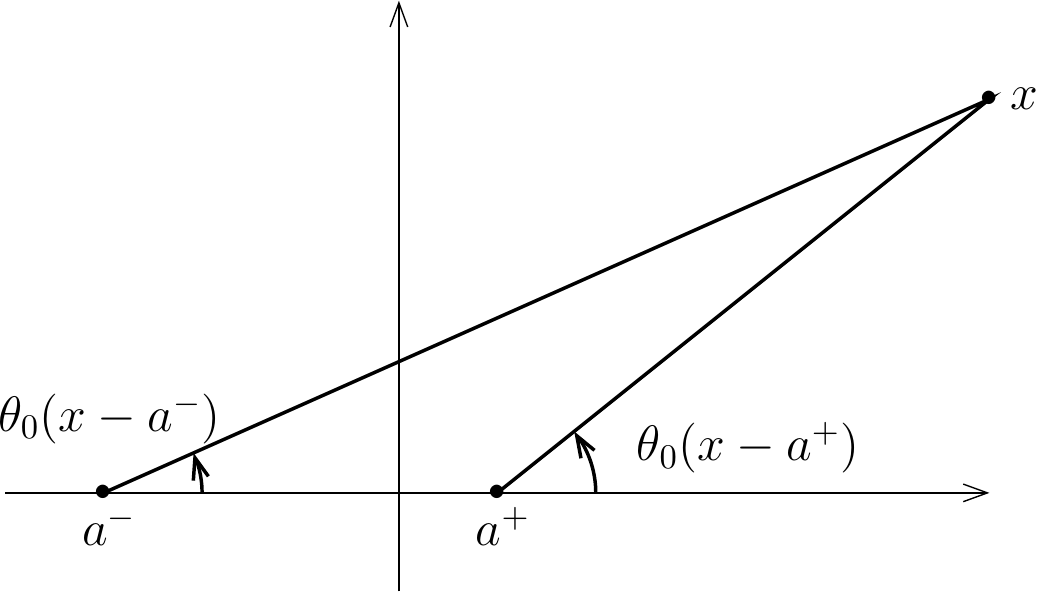}
\caption{Geometric interpretation of the function $\varphi$.}\label{figPotential}
\end{center}
\end{figure}
In particular, if $x_1>a^+_1$, 
\begin{equation*}\varphi_{a^-,a^+}(x_1,0)=\frac12\theta_0(x_1-a^+_1,0)-\frac12\theta_0(x_1-a^-_1,0)=0-0=0.\end{equation*}
The function $\varphi_{a^-,a^+}$ is smooth on
$\RR^2\setminus((-\infty,a^+_1]\times\{0\})$. Furthermore
\eqref{eqGradTheta}, together with the chain rule, gives
\begin{align*}
  \nabla\varphi_{a^-,a^+}(x)&=\frac12\frac{1}{(x_1-a^+_1)^2+x_2^2}(-x_2,x_1-a^+_1)-
  \frac12\frac{1}{(x_1-a^-_1)^2+x_2^2}(-x_2,x_1-a^-_1)\\
  &= A_{a^-,a^+}(x)
\end{align*}
for all $x \in \RR^2\setminus(-\infty,a^+_1]\times\{0\}$. 
Furthermore we have that, for any $x_1<0$,
\begin{equation*}\lim_{\eta\to 0,\eta>0} \theta_0(x_1,\eta)=\pi \quad\mbox{and}\quad \lim_{\eta\to 0,\eta>0} \theta_0(x_1,-\eta)=-\pi.\end{equation*}
This implies that 
\begin{equation*}\lim_{\eta\to 0,\eta>0} \varphi_{a^-,a^+}(x_1,\eta)= \lim_{\eta\to 0,\eta>0} \varphi_{a^-,a^+}(x_1,-\eta)=0\end{equation*}
for all $x_1<a^-_1$ and that
\begin{equation*}
\lim_{\eta\to 0,\eta>0} \varphi_{a^-,a^+}(x_1,\eta)=\frac{\pi}2 \quad \mbox{and}\quad \lim_{\eta\to 0,\eta>0}
\varphi_{a^-,a^+}(x_1,-\eta)=-\frac{\pi}2
\end{equation*}
for all $x_1\in (a^-_1,a^+_1)$. In particular, $\varphi_{a^-,a^+}$ has
a continuous extension to $\RR^2\setminus \ic(a^-,a^+)$, which we
also denote by $\varphi_{a^-,a^+}$. Since $\pot$ is smooth on
$\RR^2\setminus \ic(a^-,a^+)$, this extension is also smooth and
satisfies $\nabla \varphi_{a^-,a^+}=\pot$. Since
$\RR^2\setminus \ic(a^-_1,a^+_1)$ is connected, the uniqueness is
obvious. Furthermore, since $\theta_0\circ\sigma =-\theta_0$, we have that
$\varphi_{a^-,a^+}\circ\sigma =-\varphi_{a^-,a^+}$.
\end{proof}
\begin{lem}\label{lemGlobalGauge}
  There exists a unique smooth function
  $\psi_{a^-,a^+}:\dpp(a^-,a^+)\to \CC$ satisfying
\begin{enumerate}[\rm (i)]
		\item $\left|\psi_{a^-,a^+}\right|\equiv 1$ on $\dpp(a^-,a^+)$;
		\item $\dfrac{\nabla
                    \psi_{a^-,a^+}}{i\psi_{a^-,a^+}}=2\pot$
                  on $\dpp(a^-,a^+)$;
		\item $\psi_{a^-,a^+}(x_1,0)=-1$ for all $x_1\in
                  (a^-_1,a^+_1)$.
\end{enumerate}
Furthermore, $\psi_{a^-,a^+}$ satisfies
$\psi_{a^-,a^+}\circ \sigma =\overline \psi_{a^-,a^+}$.
\end{lem}
\begin{proof}
  For all $x\in \RR^2\setminus \ic(a^-,a^+)$, we set
  $\psi_{a^-,a^+}(x)=e^{2i\varphi_{a^-,a^+}(x)}$, where
  $\varphi_{a^-,a^+}$ is the function defined in Lemma
  \ref{lemAngle}. This function is smooth on
  $\RR^2\setminus \ic(a^-,a^+)$ and, for all
  $x \in \RR^2\setminus \ic(a^-,a^+)$,
 \begin{equation*}
	 \nabla \psi_{a^-,a^+}(x)=2i e^{2i\varphi_{a^-,a^+}(x)} \nabla \varphi_{a^-,a^+}(x)= 2i\psi_{a^-,a^+}(x)\pot(x),
 \end{equation*}
 and thus 
 \begin{equation*}
	 \frac{\nabla \psi_{a^-,a^+}(x)}{i\psi_{a^-,a^+}(x)}=2\nabla \varphi_{a^-,a^+} (x)= 2 \pot(x).
 \end{equation*}
 On the other hand, for all $x_1\in (a^-_1,a^+_1)$,
 \begin{equation*}
	 \lim_{\eta\to 0,\eta >0}\psi_{a^-,a^+}(x_1,\eta)=e^{i\pi}=-1 \mbox{ and  }
	 \lim_{\eta \to 0, \eta>0}\psi_{a^-,a^+}(x_1,-\eta)=e^{-i\pi}=-1.
 \end{equation*}
 This implies that $\psi_{a^-,a^+}$ admits a continuous extension to
 $\dpp(a^-,a^+)$, which we also denote by $\psi_{a^-,a^+}$. Since
 $\nabla\psi_{a^-,a^+} =2i\psi_{a^-,a^+} \pot$ on
 $\RR^2\setminus \ic(a^-,a^+)$, with $\pot$ smooth on $\dpp(a^-,a^+)$,
 we obtain that $\psi_{a^-,a^+}$ is of class $C^1$ on $\dpp(a^-,a^+)$,
 and then that $\psi_{a^-,a^+}$ is smooth by a bootstrap argument.
 
 Let us now prove uniqueness. Let us assume that $\widetilde \psi$ is
 a function satisfying conditions (i-iii). Then, we deduce from (ii) that
 \begin{equation*}
   \nabla\left(\frac{\widetilde \psi}{\psi_{a^-,a^+}}\right)=
   \frac{\widetilde \psi}{\psi_{a^-,a^+}}\left(\frac{\nabla
       \widetilde\psi}{\widetilde \psi}
     -\frac{\nabla \psi_{a^-,a^+}}{\psi_{a^-,a^+}}\right)
   =\frac{\widetilde \psi}{\psi_{a^-,a^+}}\left(2i\pot-2i\pot\right)=0
 \end{equation*}
 on $\dpp(a^-,a^+)$. There exists therefore $c \in \CC$ such that
 $\widetilde \psi=c \psi_{a^-,a^+}$ on $\dpp(a^-,a^+)$, and condition
 (iii) tells us that $c=1$, thus proving uniqueness.

  Finally, since
 $\varphi_{a^-,a^+} \circ \sigma =-\varphi_{a^-,a^+}$, we conclude that
 $\psi_{a^-,a^+}\circ \sigma =\overline{\psi_{a^-,a^+}}$.
\end{proof}

To simplify notation, in the following sections, we do not write
explicitly the dependence on $a^-$ and $a^+$, except for the
eigenvalues, but the objects considered depend on the position of the
two poles (so we will write $\HA$ for $H_{A_{a^-,a^+}}$, $\dpOm$ for
$\dpOm(a^-,a^+)$, etc.).

\subsection{Conjugation and symmetry}\label{sec:conjugation-symmetry}

\begin{dfn} Let us define the antilinear, antiunitary operator $K$ on
  $L^2(\Omega)$ by $Ku:=\psi \overline{u}$, where $\psi$ is the gauge
  function defined in Lemma \ref{lemGlobalGauge}. We say that a
  function $u\in L^2(\Omega)$ is \emph{$K$-real} if $Ku=u$. We denote by
  $L^2_K(\Omega)$ the set of $K$-real functions.
\end{dfn}

\begin{lem}\label{lemScalarProd}
	If $u$ and $v$ are in $L^2_K(\Omega)$, $\langle u,v\rangle\in \RR$. 
\end{lem}

\begin{proof}
	We have \begin{equation*}\langle u,v\rangle=\int_{\Omega}u\overline v \,dx= \int_{\Omega}\overline\psi u\psi\overline v\,dx=\int_{\Omega}\overline{\psi \overline u}\psi\overline v\,dx=\int_{\Omega}\overline u v \,dx=\overline{\langle u,v\rangle}.\qedhere\end{equation*}
\end{proof}

\begin{rem} The set $L^2_K(\Omega)$ is not a subspace of the complex
  vector space $L^2(\Omega)$, because multiplication by a complex
  number does not preserve $K$-real functions. However, multiplication
  by a real number does preserve these functions, and therefore
  $L^2_K(\Omega)$ is a real vector space. Moreover, Lemma
  \ref{lemScalarProd} shows that the restriction to $L^2_K(\Omega)$ of
  the complex scalar product on $L^2(\Omega)$ is a real scalar
  product. Therefore, $L^2_K(\Omega)$ is a real Hilbert space.
\end{rem} 
\begin{lem}\label{lemComK}
The antilinear operator $K$ preserves the
  domain of $\HA$, and $\HA\circ K=K\circ \HA$.
\end{lem}

\begin{proof}
Let us begin by considering $u\in C^{\infty}_c(\dpOm)$. We have 
\begin{equation*}
  (i\nabla+A)(Ku)=(i\nabla+A)(\psi \overline
  u)=i\psi\nabla \overline u +\psi \overline{u} A + i\overline{u}\nabla \psi.
\end{equation*}
Since $\nabla \psi =2i\psi A$, we obtain
\begin{equation*}
  (i\nabla+A)(Ku)=i\psi \nabla \overline u-\psi \overline u A=-\psi \overline{(i\nabla+A)u}.
\end{equation*}
As a consequence, for any $v\in C^{\infty}_c(\dpOm)$,
\begin{align*}
  \int_{\Omega}H_AKu\,\overline v \,dx &=\int_{\Omega}
  (i\nabla+A)(Ku)\cdot\overline{(i\nabla+A)v}\,dx=-\int_{\Omega}\psi\overline{(i\nabla+A)u}\cdot
  \overline{(i\nabla+A)v}\,dx\\
&= \int_{\Omega}(i\nabla+A)(Kv)\cdot
  \overline{(i\nabla+A)u}\,dx
=\int_{\Omega}Kv\,
  \overline{(i\nabla+A)^2u}\,dx\\
&= \int_{\Omega}\psi \overline v\,
  \overline{(i\nabla+A)^2u}\,dx=\int_{\Omega}KH_Au\,\overline v \,dx.
\end{align*}
We therefore have $H_AKu=KH_Au$ for all $u\in
C^{\infty}_c(\dpOm)$. The conclusion follows by density.
\end{proof}

We deduce  from  Lemma \ref{lemComK} that the eigenspaces of $H_A$ are
stable under the action of $K$. This implies that we can find a basis
of $L^2(\Omega)$ formed by $K$-real eigenfunctions of $H_A$. We also interpret this in another
way: $L^2_K(\Omega)$ is stable under the action of $H_A$ and the
restriction of $H_A$ to $L^2_K(\Omega)$ has the same spectrum as
$H_A$.

We now want to study the consequence of the fact that $\Omega$ is
symmetric with respect to $\axis$ on the operator $H_A$. We therefore
define the antiunitary antilinear operator $\Sigma^c$, acting on
$L^2(\Omega)$, by $\Sigma^c u:=\overline{u}\circ \sigma$.

\begin{lem} \label{lemComSigma}
	The antilinear operator $\Sigma^c$ preserves the domain of $\HA$, and $\HA\circ \Sigma^c=\Sigma^c\circ \HA$. Furthermore, $\Sigma^c\circ K=K\circ \Sigma^c$. 
\end{lem}

\begin{proof}
  The second point is clear: if $u\in L^2(\Omega)$,
  \begin{equation*}
    \left(\Sigma^cK\right)u(x)=\overline
    {\left(Ku\right)(\sigma(x))}=\overline{\psi(\sigma(x))\overline{u(\sigma(x))}}=\psi(x)\overline
    {\overline{ u(\sigma(x))}}=(K\Sigma^c)u(x).
\end{equation*}
  To prove the first point, we begin by considering $u\in C^{\infty}_c(\dpOm)$. We have
\[
(i\nabla+A)(\Sigma^c u)=i\nabla (\overline u \circ
          \sigma)+(\overline u) \circ \sigma A= \\(i(\nabla \overline
          u)\circ \sigma)S - ((\overline u A)\circ \sigma)S=
          -\overline{\left((i\nabla +A)u\circ \sigma\right)}S.
\]
	For any $v\in C^{\infty}_c(\dpOm)$, we have 
\[
\int_{\Omega}(H_A\Sigma^c u)\, \overline
v\,dx=\int_{\Omega}(i\nabla+A)(\Sigma^c
u)\cdot\overline{(i\nabla+A)v}\,dx=
\int_{\Omega}-\overline{\left((i\nabla +A)u\circ
    \sigma\right)}S\cdot\overline{(i\nabla+A)v}\,dx.
\]
After the change of variable $x=\sigma(y)$, and using the fact that $S$ is symmetric, we find
\begin{align*}
  \int_{\Omega}
  (H_A\Sigma^c u)\, 
  \overline v\,dx&=\int_{\Omega}\overline{(i\nabla
                   +A)u}\cdot-\overline{\left((i\nabla+A)v\circ
                   \sigma\right)}S\,dy\\
                 &=\int_{\Omega}\overline{(i\nabla +A)u}\cdot(i\nabla+A)(\Sigma^c v)\,dy=\int_{\Omega} \overline{H_Au}\, \Sigma^c v\,dy=\int_{\Omega} \overline{H_Au}\,  (\overline v\circ \sigma)\,dy.
\end{align*}
We now do the reverse change of variable $y=\sigma(x)$, thus obtaining
\begin{equation*}
  \int_{\Omega}
  (H_A\Sigma^c u)\, 
  \overline v\,dx=\int_{\Omega}\left(\overline{H_Au}\circ \sigma\right) \overline v \,dx=\int_{\Omega}\left(\Sigma^cH_Au\right) \overline v \,dx.
\end{equation*}
We therefore have $H_A\Sigma^cu=\Sigma^cH_Au$ for all $u\in C^{\infty}_c(\dpOm)$. The conclusion follows by density.
\end{proof}

The second point of Lemma \ref{lemComSigma} implies that $L^2_K(\Omega)$ is stable under the action of $\Sigma^c$. If we write 
\begin{equation*}L^2_{K,\Sigma}(\Omega):= L^2_K(\Omega)\cap \mbox{ker}(\Sigma^c-Id)\end{equation*}
and
\begin{equation*}L^2_{K,a\Sigma}(\Omega):= L^2_K(\Omega)\cap \mbox{ker}(\Sigma^c+Id),\end{equation*}
we observe that every function $u\in L^2_K(\Omega)$ can be decomposed
as 
\begin{equation}\label{eq:15}
u=\frac12(u+\bar u\circ\sigma)+\frac12(u-\bar u\circ\sigma),
\end{equation}
so
that we
have the orthogonal decomposition
\begin{equation}\label{eq:orthogonal}
L^2_K(\Omega)=L^2_{K,\Sigma}(\Omega)\oplus L^2_{K,a\Sigma}(\Omega).
\end{equation}
The first point of Lemma \ref{lemComSigma} implies that $H_A$ leaves
the spaces $L^2_{K,\Sigma}$ and $L^2_{K,a\Sigma}$ invariant. We can
therefore define the operators $H_{A,\Sigma}$ and $H_{A,a\Sigma}$,
restrictions of $H_A$ to $L^2_{K,\Sigma}(\Omega)$ and
$L^2_{K,a\Sigma}(\Omega)$ respectively. The spectrum of $H_A$ is the
reunion (counted with multiplicities) of the spectra of $H_{A,\Sigma}$
and $H_{A,a\Sigma}$.

\subsection{Spectral equivalence to the Laplacian with mixed boundary conditions}\label{sec:spectr-equiv-lapl}

Let us consider the \emph{real} Hilbert space
$L^2_{\RR,\sigma}(\Omega)$ consisting of the real valued
$L^2$-functions $u$ on $\Omega$ such that $u\circ\sigma=u$.

Let us consider the operator $H_{NDN}$  on
  $L^2_{\RR,\sigma}(\Omega)$ defined as the Friedrichs
extension of the differential
operator $-\Delta$ acting on the domain $\{u\in C^{\infty}_{\rm
  c}(\Omega\setminus \mathcal I_c,\RR):u\circ\sigma=u\}$, the space of
real valued smooth functions with compact
support in $\Omega\setminus \mathcal I_c$
symmetric with respect to the axis $x_2=0$. 
The domain of $H_{NDN}$ is
 then given by $\{u\in H^1_0(\Omega\setminus\mathcal
I_c):u\circ\sigma=u\text{ and }\Delta\big|_{\Omega\setminus\mathcal
I_c} u\in L^2_{\RR,\sigma}(\Omega)\}$, being $\Delta\big|_{\Omega\setminus\mathcal
I_c} $ the distributional Laplacian in $\Omega\setminus\mathcal
I_c$. $H_{NDN}$ is a  symmetric, positive, and self-adjoint operator on
  $L^2_{\RR,\sigma}(\Omega)$.  We denote by
  $(\lambda_k^{NDN}(a^+,a^-))_{k\ge 1}$ its eigenvalues.

  In a similar way, we consider the operator $H_{DND}$ on
  $L^2_{\RR,\sigma}(\Omega)$ defined as the Friedrichs extension of
  $-\Delta$ acting on
  $\{u\in C^{\infty}_{\rm c}(\Omega\setminus (\mathcal R\setminus
  \mathcal I_c),\RR):u\circ\sigma=u\}$.
  The domain of $H_{DND}$ is then given by
  $\{u\in H^1_0(\Omega\setminus (\mathcal R\setminus \mathcal I_c)):
  u\circ\sigma=u\text{ and }\Delta\big|_{\Omega\setminus (\mathcal R\setminus \mathcal I_c)}
  u\in L^2_{\RR,\sigma}(\Omega)\}$,
  being
  $\Delta\big|_{\Omega\setminus (\mathcal R\setminus \mathcal I_c)}$
  the distributional Laplacian in
  $\Omega\setminus (\mathcal R\setminus \mathcal I_c)$. $H_{DND}$ is a
  symmetric, positive, and self-adjoint operator on
  $L^2_{\RR,\sigma}(\Omega)$.  We denote by
  $(\lambda_k^{DND}(a^+,a^-))_{k\ge 1}$ its eigenvalues.

\begin{rem}
 Let us consider the upper-half domain associated with $\Omega$
\begin{equation*}
  \Omega^{uh}:=\Omega\cap\{(x_1,x_2)\in \RR^2\,:\, x_2>0\}.
\end{equation*}
We have that $\partial\Omega^{uh}:=\Gamma^{uh}\cup\Gamma_0$,
with
$\Gamma^{uh}:=\partial\Omega\cap\{(x_1,x_2)\in \RR^2\,:\, x_2>0\} $
and 
$\Gamma^0:=\overline{\Omega}\cap \axis$. 
We additionally define $\Gamma_c^0=\Gamma^0\cap \ic$. 

  We notice that, if $\Omega$ has smooth boundary, then the operator
  $H_{NDN}$ can be identified with the Neumann-Dirichlet-Neumann
  Laplacian on $\Omega^{uh}$ denoted by $-\Delta^{NDN}$ and defined as
  the Laplacian on $\Omega^{uh}$ with Dirichlet boundary condition on
  $\Gamma^{uh}\cup \Gamma^0_c$ and Neumann boundary condition on
  $\Gamma^0\setminus\Gamma^0_c$, see Figure
  \ref{figNDN}.

In a similar way, if $\Omega$ has smooth boundary, then the operator
  $H_{DND}$ can be identified with the  Dirichlet-Neumann-Dirichlet Laplacian
        $-\Delta^{DND}$ defined as the Laplacian on $\Omega^{uh}$
        with Dirichlet boundary condition on
        $\Gamma^{uh}\cup(\Gamma^0\setminus\Gamma^0_c)$ and Neumann
        boundary condition on $\Gamma^0_c$, see
        Figure \ref{figDND}.
\end{rem}

\begin{figure}[h]
\begin{center}
  \subfigure[Neumann-Dirichlet-Neumann boundary
  conditions\label{figNDN}]{\includegraphics[width=5cm]{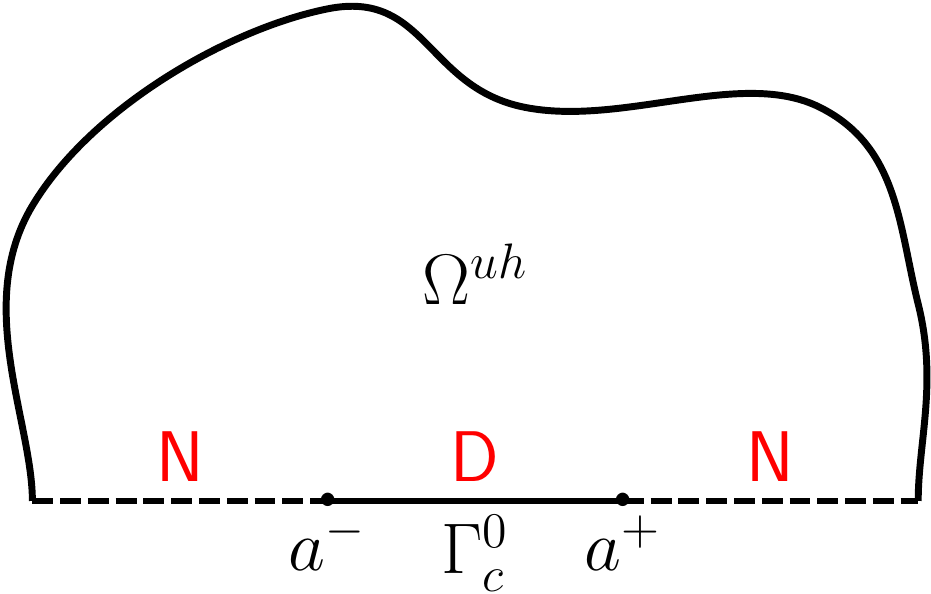}}
  \quad\subfigure[Dirichlet-Neumann-Dirichlet boundary
  conditions\label{figDND}]{\includegraphics[width=5cm]{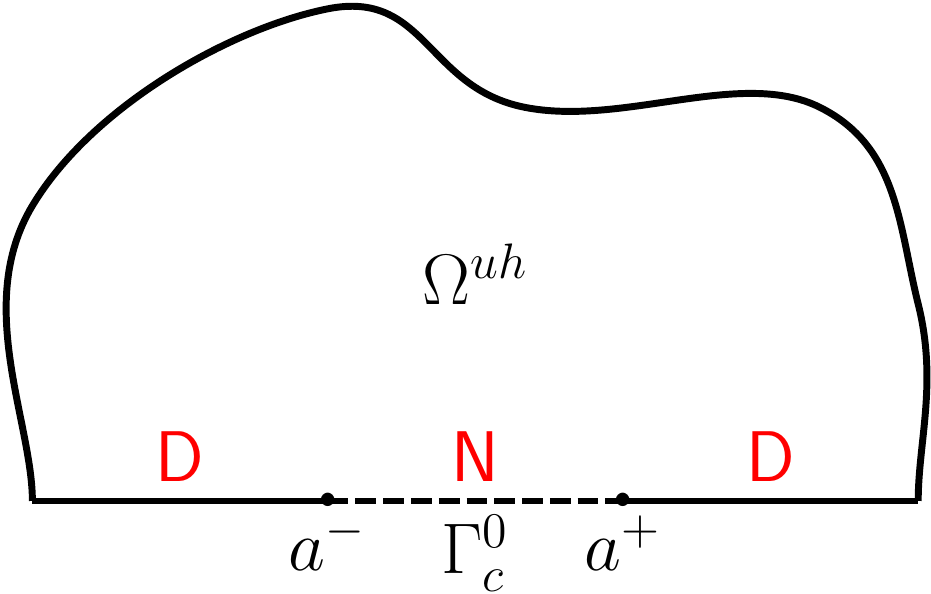}}
\caption{Eigenvalue problems with mixed boundary conditions in $\Omega^{uh}$.}
\end{center}
\end{figure}

The main result of this section if the following equivalence of
$H_{NDN}$ with $H_{A,\Sigma}$ and of $H_{DND}$ with $H_{A,a\Sigma}$.

\begin{prop}\label{propIso}
  The operator $H_{NDN}$ is unitarily equivalent to $H_{A,\Sigma}$ and
  the operator $H_{DND}$ is unitarily equivalent to $H_{A,a\Sigma}$.
\end{prop}

Before proving Proposition \ref{propIso}, we observe that a direct
consequence of Proposition \ref{propIso} combined with the discussion
in \S \ref{sec:conjugation-symmetry} is the following isospectrality
result.

\begin{cor}\label{cor:iso} The sequence $(\lambda_k(a^+,a^-))$ is the reunion, counted with multiplicities, of the sequences $(\lambda_k^{NDN}(a^+,a^-))_{k\ge 1}$ and $(\lambda_k^{DND}(a^+,a^-))_{k\ge 1}$.
\end{cor}

We divide the proof of Proposition \ref{propIso} into two lemmas. The first gives information on the nodal set of functions in $L^2_{K,\Sigma}(\Omega)$ or $L^2_{K,a\Sigma}(\Omega)$.

\begin{lem}\label{lemZeroSet}
  If $u\in L^2_{K,\Sigma}(\Omega)\cap C(\dpOm)$, then $u\equiv 0$ on
  $\Omega \cap \ic$.  If $u\in L^2_{K,a\Sigma}(\Omega)\cap C(\dpOm)$,
  then $u\equiv0$ on $(\Omega\cap\axis)\setminus \ic$.
\end{lem}

\begin{proof}
  Since $Ku=u$, we have $\overline{u}(x_1,0)=u(x_1,0)$ if $x_1<a^-_1$
  of $x_1>a^+_1$, and $\overline{u}(x_1,0)=-u(x_1,0)$ and if
  $x_1\in(a^-_1,a^+_1)$.
	
  If $\Sigma^cu=u$, $\overline{u}(x_1,0)=u(x_1,0)$ for
  all $x_1$, and therefore $u(x_1,0)=0$ if $x_1\in(a^-_1,a^+_1)$. In
  the case where $\Sigma^cu=-u$, $\overline{u}(x_1,0)=-u(x_1,0)$ for
  all $x_1$, and therefore $u(x_1,0)=0$ if $x_1<a^-_1$ or $x_1>a^+_1$.
\end{proof}

The second Lemma gives a unitary operator of similarity, and proves the
two isospectrality results of Proposition \ref{propIso}.

\begin{lem}\label{l:unitmap}
If $u\in L^2(\Omega)$, we define
\begin{equation*}
  U_\sigma u:=e^{-i\varphi} u, \quad U_{a\sigma} u:=
\begin{cases}
e^{-i\varphi} u,&\text{in }\Omega^{uh},\\
-e^{-i\varphi} u,&\text{in }\Omega\setminus\Omega^{uh}
\end{cases}
\end{equation*} 
where $\varphi$ is the function defined in Lemma \ref{lemAngle}.  We have the following properties:
\begin{enumerate}[\rm (i)]
\item $U_\sigma$ defines a one-to-one and unitary mapping from
  $L^2_{K,\Sigma}(\Omega)$ to $L^2_{\RR,\sigma}(\Omega)$ and  
$U_{a\sigma}$ defines a one-to-one and unitary mapping from
    $L^2_{K,a\Sigma}(\Omega)$ to $L^2_{\RR,\sigma}(\Omega)$;
\item $U_\sigma $ maps the domain of $H_{A,\Sigma}$ to the domain of
  $H_{NDN}$ and $U_\sigma\circ H_{A,\Sigma}=H_{NDN}\circ U_{\sigma}$;
\item $U_{a\sigma}$ maps the domain of $H_{A,a\Sigma}$ to the domain of
  $H_{DND}$ and $U_{a\sigma}\circ H_{A,a\Sigma}=H_{DND}\circ U_{a\sigma}$.
\end{enumerate}
\end{lem}

\begin{proof}
  Let us first check that for all $u\in L^2_K(\Omega)$,
  $\overline{U_\sigma u}=U_\sigma u$ and $\overline{U_{a \sigma} u}=U_{a,\sigma} u$. Indeed, for $x\in \Omega$,
\[
  \overline{e^{-i\varphi(x)}u(x)}=e^{-i\varphi(x)}e^{2i\varphi(x)}\overline
  u(x)
  =e^{-i\varphi(x)}(Ku)(x)=e^{-i\varphi(x)}u(x).
\]
If $u\in L^2_{K,\Sigma}(\Omega)$, then $U_\sigma u(\sigma(x))=e^{-i\varphi(\sigma(x))}u(\sigma(x))=
e^{i\varphi(x)}\overline{u(x)}=\overline{U_\sigma
  u(x)}=U_\sigma u(x)$ so that $U_\sigma u\in L^2_{\RR,\sigma}(\Omega)$.
If $u\in L^2_{K,a\Sigma}(\Omega)$, then $U_{a\sigma}
u(\sigma(x))=U_{a\sigma} u(x)$ so that 
$U_{a\sigma} u\in L^2_{\RR,\sigma}(\Omega)$.

Furthermore, if $u\in L^2_{K}(\Omega)$ then $|U_\sigma u|=|u|$ and $|U_{a\sigma} u|=|u|$, therefore
\begin{equation*}
  \int_{\Omega}|u|^2\,dx=\int_{\Omega}\left|U_\sigma u\right|^2\,dx=\int_{\Omega}\left|U_{a\sigma} u\right|^2\,dx.
\end{equation*}
Finally, if $v \in L^2_{\RR,\sigma}(\Omega)$, a direct computation shows
that the function $u_\Sigma$ defined on $\Omega$ by
\begin{equation*}
  u_{\Sigma}(x):=e^{i\varphi(x)}v(x)
\end{equation*}
	is in $L^2_{K,\Sigma}(\Omega)$ and that $U_\sigma
        u_{\Sigma}=v$. 
This shows that $U_\sigma$ defines a one-to-one map from
  $L^2_{K,\Sigma}(\Omega)$ to $L^2_{\RR,\sigma}(\Omega)$.
In the same way, the function $u_{a\Sigma}$ defined on $\Omega$ by
\begin{equation*}
  u_{a\Sigma}(x):=e^{i\varphi(x)}v(x) \mbox{ for } x \in \Omega^{uh}
\end{equation*}
	and
\begin{equation*}
  u_{a\Sigma}(x):=-e^{i\varphi(x)}v(x)\mbox{ for } x \in \Omega\cap \{(x_1,x_2)\in \RR^2\,:\,x_2<0\}
\end{equation*}
is in $L^2_{K,a\Sigma}(\Omega)$ and $U_{a\sigma}u_{a\Sigma}=v$. This
shows that
$U_{a\sigma}$ defines a one-to-one map from
    $L^2_{K,a\Sigma}(\Omega)$ to $L^2_{\RR,\sigma}(\Omega)$. We have proved point (i).
	
To prove point (ii), let us begin by considering
$u\in C^{\infty}_c(\dpOm)\cap L^2_{K,\Sigma}(\Omega)$. According to
Lemma \ref{lemZeroSet}, $u\equiv 0$ on $\Omega \cap \ic$. Then
$U_\sigma u\in H^1_0(\Omega\setminus\mathcal I_c)$ (for this it is
crucial that $u$ vanishes on $\mathcal I_c$ since $e^{-i\varphi}$
jumps across $\mathcal I_c$) and
\begin{equation*}
  (i\nabla)(U_\sigma u)=i\nabla(e^{-i\varphi}
  u)=e^{-i\varphi}(i\nabla+\nabla \varphi) u=e^{-i\varphi}(i\nabla+A)
  u
\quad\text{in }\Omega\setminus\mathcal I_c. 
\end{equation*}
We observe that any function $u$ in the domain of the operator
$H_{A,\Sigma}$ can be approximated in the form domain norm by
functions in  $C^{\infty}_c(\dpOm)\cap L^2_{K,\Sigma}(\Omega)$. To
this aim, we can first take a sequence of functions $u_n\in
C^{\infty}_c(\dpOm)$ converging to $u$ in the form domain norm;
then we take the sequence $v_n=\frac14(u_n+\overline{u_n\circ
  \sigma}+K(u_n+\overline{u_n\circ \sigma})$ which stays in
$C^{\infty}_c(\dpOm)\cap L^2_{K,\Sigma}(\Omega)$ and converges to $u$
in the form domain norm
thanks to the validity of the Hardy type inequality 
\[
\|Aw\|_{L^2(\Omega)}\leq C(a,\Omega)\|(i\nabla+A)w\|_{L^2(\Omega)}
\]
which holds for every $w\in C^{\infty}_c(\dpOm)$ and for some
$C(a,\Omega)>0$ depending on $a$ and $\Omega$.

Then we conclude that, for every $u$ in the domain of the operator
$H_{A,\Sigma}$, $U_\sigma u\in H^1_0(\Omega\setminus\mathcal I_c)$ and
\begin{equation*}
  (i\nabla)(U_\sigma u)=e^{-i\varphi}(i\nabla+A)
  u
\quad\text{in }\Omega\setminus\mathcal I_c. 
\end{equation*}
Furthermore, for every
$w\in C^\infty_{c}(\Omega\setminus\mathcal I_c)$
 \begin{equation}\label{eq:16}
\int_{\Omega\setminus\mathcal I_c}\nabla(U_\sigma u)\cdot\nabla w\,dx=
\int_{\dpOm}(i\nabla+A)u\cdot\overline{(i\nabla
  +A)(e^{i\varphi}w)}\,dx. 
\end{equation}
Since $u$ is in the domain of the  Friedrichs
extension of the differential
operator $(i\nabla+A)^2$, we conclude that 
 \[
\left|\int_{\Omega\setminus\mathcal I_c}\nabla(U_\sigma u)\cdot\nabla
  w\,dx\right|\leq {\rm const}\|w\|_{L^2(\Omega)}
\]
for every $w\in C^\infty_{c}(\Omega\setminus\mathcal I_c)$, thus
implying that $U_\sigma u$ stays in the domain of $H_{NDN}$. 
Moreover, by density and \eqref{eq:16} we conclude that 
\begin{equation*}
  U_\sigma(H_{NDN} u)=e^{-i\varphi}(i\nabla+A)^2u
\end{equation*}
completing the proof of (ii).
	
The proof of (iii) can be obtained in a similar way, observing that 
any $u\in C^{\infty}_c(\dpOm)\cap L^2_{A,a\Sigma}(\Omega)$ vanishes on
$\mathcal R\setminus\mathcal I_c$; hence $U_{a\sigma}u\in
H^1_0(\Omega\setminus(\mathcal R\setminus\mathcal I_c))$ (for this it is
crucial that $u$ vanishes on $\mathcal R\setminus \mathcal I_c$ since
$\mathop{\rm sign}(x_2)e^{-i\varphi}$
jumps across $\mathcal R\setminus \mathcal I_c$).
\end{proof}

\subsection{Proof of Theorem \ref{t:ab}}

Combining the isospectrality result of Corollary \ref{cor:iso} with
Theorem \ref{t:asynull} we can now prove Theorem \ref{t:ab}.

\begin{proof}[Proof of Theorem \ref{t:ab}]
 It is known from \cite{lena2015} that $\lambda_{N}^a\to
 \lambda_{N}(\Omega)$; in particular the continuity result of
 \cite{lena2015} implies that, since $\lambda_N(\Omega)$ is simple,
 then also $\lambda_{N}^a$ is simple. It is not restrictive to assume that $u_N$ is
 real valued.
It is easy to prove that, for all $a>0$ small, there exists $u_N^a$
eigenfunction of $(i\nabla +\pot)^2$ associate to $\lambda_N^a$ such
that 
 \begin{equation}\label{eq:21}
u_N^a\to u_N\quad \text{in }C^2_{\rm loc}(\Omega\setminus\{0\},\CC)
\end{equation}
as $a\to0^+$; moreover it is possible to choose $u_N^a\in
L^2_{K_{a^-,a^+}}(\Omega)$ (otherwise take 
$\frac12(u_N^a+K_{a^-,a^+}(u_N^a))$ which is a $K_{a^-,a^+}$-real
eigenfunction for $\lambda_{N}^a$ still converging to $u_N$).
 
The orthogonal decomposition \eqref{eq:15} and the simplicity of
$\lambda_N^a$ imply that either $u_N^a\in 
L^2_{K,\Sigma}(\Omega)$ (and then, by Lemma \ref{lemZeroSet},
$u_N^a\equiv 0$ on $[-a,a]\times\{0\}$)
 or $u_N^a\in L^2_{K,a\Sigma}(\Omega)$ (and then
$u_N^a\equiv 0$ on $(\RR\setminus(-a,a))\times\{0\}$).
If $u_N^a\equiv 0$ on $(\RR\setminus(-a,a))\times\{0\}$, then
\eqref{eq:21} would imply that $u_N\equiv 0$ on $\RR\times\{0\}$ thus
contradicting the assumption $\alpha\neq0$. Hence we have that
necessarily  $u_N^a\in 
L^2_{K,\Sigma}(\Omega)$. Then $\lambda_N^a$ is an eigenvalue of
$H_{A_{a^-,a^+},\Sigma}$ and, by Proposition \ref{propIso}, of
$H_{NDN}$. Therefore 
\[
\lambda_N^a=\lambda_N(\Omega\setminus([-a,a]\times\{0\}))
\]
and the conclusion follows applying Theorem \ref{t:asynull}.
\end{proof}

\appendix
 \section{Proof of Theorem \ref{propL1intro}}

 Since our setting is a little
 different from \cite{Courtois1995Holes} (which considers  manifolds
 without boundary) and Theorem \ref{propL1intro} is quite hidden in
 the arguments of  \cite{Courtois1995Holes}, we think it  is
 worthwhile giving in this appendix a proof of Theorem
 \ref{propL1intro}. Our approach is different from the one used in
 \cite{Courtois1995Holes}.
 It relies on the spectral theorem to estimate how closely approximate
 eigenvalues and eigenfunctions
 approach the true one.

Let us begin with the following crucial estimate, which is the analogue of \cite[Lemma 3.2]{Courtois1995Holes}.
\begin{lem}\label{lemL2Norm}
Let  $\Omega\subset\RR^n$ be a bounded, connected, and open set.
   If $(K_{\eps})_{\eps>0}$ is a family of compact sets contained in
   $\Omega$ 
  concentrating to a compact set $K$ with $\mbox{\rm
    Cap}_{\Omega}K=0$, then for every $f\in H_0^1(\Omega)$
\begin{equation*}
  \int_{\Omega}|V_{K_{\eps},f}|^2\,dx =o\left(\mbox{\rm Cap}_{\Omega}(K_{\eps},f)\right)
  \quad \text{ as }\eps \to 0.
\end{equation*}
\end{lem}

\begin{proof}
 Let us assume by contradiction that there exists a
  sequence $\eps_n\to 0$ and a constant $C>0$ such that
\begin{equation*}
  \int_{\Omega}|V_{K_{\eps_n},f}|^2\,dx \ge \frac1C\mbox{\rm Cap}_{\Omega}(K_{\eps_n},f).
\end{equation*}
We set
\begin{equation*}
	W_n:=\frac1{\left\|V_{K_{\eps_n},f}\right\|_{L^2(\Omega)}}V_{K_{\eps_n},f}.
\end{equation*}
We have 
\begin{equation*}
 \left\|W_n\right\|_{L^2(\Omega)}=1
\end{equation*}
and
\begin{equation*}
 \left\|\nabla W_n\right\|_{L^2(\Omega)}^2=\frac1{\left\|V_{K_{\eps_n},f}\right\|_{L^2(\Omega)}^2}\mbox{\rm Cap}_{\Omega}(K_{\eps_n},f)\le C.
\end{equation*}
By weak compactness of the unit ball of $H^1_0(\Omega)$ and
compactness of the inclusion $H^1_0(\Omega)\subset L^2(\Omega)$, there
exists an increasing sequence of integers $(n_k)_{k\ge 1}$ and a
function $W\in H^1_0(\Omega)$ such that
$\left(W_{n_k}\right)_{k\ge 1}$ converges to $W$ when $k$ goes to
$+\infty$, weakly in $H^1_0(\Omega)$ and strongly in $L^2(\Omega)$.
We have that $\left\|W\right\|_{L^2(\Omega)}=1$ and $\Delta W=0$ in
$\Omega \setminus K$ in a weak sense.  This last equation implies that
$W$ is harmonic in $\Omega$ (since $\mbox{\rm Cap}_{\Omega}K=0$), and
therefore that $W$ is identically $0$. We
have reached a contradiction and proved the lemma.
\end{proof}

\begin{cor}
	\label{corNormL2}
	If $(K_{\eps})_{\eps>0}$ is a family of compact sets contained in $\Omega$
	concentrating to a compact set $K\subset\Omega$ 
with $\mbox{\rm Cap}_{\Omega}K=0$, then, for any $f\in H^1_0(\Omega)\cap L^\infty(\Omega)$,
	\[
\int_{\Omega}|V_{K_\eps,f}|^2\,dx
=o\left(\mbox{\rm Cap}_{\Omega}(K_\eps)\right) \qquad \text{ as
}\eps\to 0.
\] 
\end{cor}
\begin{proof}
  By the maximum principle for harmonic functions in
  $\Omega\setminus K_\eps$ we have that
\[
\left|V_{K_\eps,f}\right|\le \left(\max_{\Omega}|f|\right)
  V_{K_\eps}.
\]
  Hence
    $\int_{\Omega}|V_{K_\eps,f}|^2\,dx\leq
    \left(\max_{\Omega}|f|\right)^2\int_{\Omega}|V_{K_\eps}|^2dx$
    and the conclusion follows from Lemma \ref{lemL2Norm} (with
    $f=\eta_{K}$).
\end{proof}

We are now in position to prove Theorem \ref{propL1intro}.

\begin{proof}[Proof of Theorem \ref{propL1intro}]
 For $\eps>0$, we denote by $-\Delta_{\eps}$ the Dirichlet Laplacian on $\Omega\setminus K_{\eps}$. More precisely, $-\Delta_{\eps}$ is the self-adjoint operator obtained from the restriction of the quadratic form
 \[q(u)=\int_{\Omega}\left|\nabla u\right|^2\,dx\] to
 $H^1_0(\Omega\setminus K_{\eps})$ through the Friedrichs' extension
 procedure (see for instance \cite[Theorem X.23]{ReeSim75}).

 To simplify notation, we write
 $\lambda_{\eps}=\lambda_N(\Omega\setminus K_{\eps})$,
 $c_{\eps}=\mbox{\rm Cap}_{\Omega}(K_\eps,u_N)$,
 $V_{\eps}=V_{K_{\eps},u_N}$, and we denote by $q$ both the quadratic
 form defined above and the associated bilinear form. We write
 $\psi_{\eps}=u_N-V_{\eps}$.  Let us note that by definition of the
 potential $V_{\eps}$, $\psi_{\eps}$ is the orthogonal projection of
 $u_N$ on $H^1_0(\Omega\setminus K_{\eps})$, in the space
 $H^1_0(\Omega)$ endowed with the scalar product $q$. For any
 $\varphi\in H_0^1(\Omega\setminus K_{\eps})$,
\begin{align*}
  q(\psi_{\eps},\varphi)-\lambda_N(\Omega)\langle \psi_{\eps},\varphi\rangle_{L^2(\Omega)}&=q(u_N,\varphi)-\lambda_N(\Omega)\langle \psi_{\eps},\varphi\rangle_{L^2(\Omega)}\\
  &= \lambda_N(\Omega)\langle
  u_N,\varphi\rangle_{L^2(\Omega)}-\lambda_N(\Omega)\langle
  \psi_{\eps},\varphi\rangle_{L^2(\Omega)}=\lambda_N(\Omega)\langle
  V_{\eps},\varphi\rangle_{L^2(\Omega)}.
\end{align*}
This means that $\psi_{\eps}$ is in the domain of the operator $-\Delta_{\eps}$ and that 
\begin{equation}
	\label{eq:Remainder}
	(-\Delta_{\eps}-\lambda_N(\Omega))\psi_{\eps}=\lambda_N(\Omega)V_{\eps}.
\end{equation}
According to Lemma \ref{lemL2Norm},
$\left\|V_{\eps}\right\|_{L^2(\Omega)}=o\big(c_{\eps}^{1/2}\big)$ as $\eps\to0^+$, so that 
\begin{equation*}
	\left\|(-\Delta_{\eps}-\lambda_N(\Omega))\psi_{\eps}\right\|_{L^2(\Omega)}=o\big(c_{\eps}^{1/2}\big)
\end{equation*}
as $\eps\to0^+$.
From the spectral theorem (see for instance \cite[Proposition 8.20]{Helffer2013}), we get
\begin{equation*}
  \mbox{\rm
    dist}\left(\lambda_N(\Omega),\sigma(-\Delta_{\eps})\right)\le 
  \frac{\left\|(-\Delta_{\eps}-\lambda_N(\Omega))\psi_{\eps}\right\|_{L^2(\Omega)}}{\left\|\psi_{\eps}\right\|_{L^2(\Omega)}}
  =o\big(c_{\eps}^{1/2}\big),\quad\text{as }\eps\to0^+,
\end{equation*}
where $\sigma(-\Delta_{\eps})$ is the spectrum of the self-adjoint
operator $-\Delta_{\eps}$. We recall that $\lambda_{\eps}\to
\lambda_N(\Omega)$ as $\eps\to 0^+$: this an immediate corollary of
\cite[Theorem 2.3]{RT}. Since $\lambda_N(\Omega)$ is assumed to be
simple, $\lambda_{\eps}$ is simple for $\eps>0$ small enough, and
\begin{equation*}
	\left|\lambda_{\eps}-\lambda_N(\Omega)\right|=o\big(c_{\eps}^{1/2}\big)\quad\text{as
        }\eps\to0^+.
\end{equation*}
Let us now denote by $\Pi_{\eps}$ the orthogonal projection from $L^2(\Omega)$ onto the one-dimensional eigenspace associated with $\lambda_{\eps}$, and let us write $\widetilde{u}_{\eps}:=\psi_{\eps}-\Pi_{\eps}\psi_{\eps}$. We have
\begin{equation*}
	\left(-\Delta_{\eps}-\lambda_{\eps}\right)\Pi_{\eps}\psi_{\eps}=0,
\end{equation*}
and therefore 
\begin{equation*}
	\left(-\Delta_{\eps}-\lambda_{\eps}\right)\widetilde{u}_{\eps}=\left(-\Delta_{\eps}-\lambda_{\eps}\right)\psi_{\eps}.
\end{equation*}
Since 
\begin{equation*}
	\left\|\left(-\Delta_{\eps}-\lambda_{\eps}\right)\psi_{\eps}\right\|_{L^2(\Omega)}\le \left|\lambda_N(\Omega)-\lambda_{\eps}\right|\left\|\psi_{\eps}\right\|_{L^2(\Omega)}+\left\|(-\Delta_{\eps}-\lambda_N(\Omega))\psi_{\eps}\right\|_{L^2(\Omega)},
\end{equation*}
we obtain
\begin{equation*}
\left\|\left(-\Delta_{\eps}-\lambda_{\eps}\right)\widetilde{u}_{\eps}\right\|_{L^2(\Omega)}=o\big(c_{\eps}^{1/2}\big)
\quad\text{as
        }\eps\to0^+.
\end{equation*}
Let us denote by $\mathcal{K}_{\eps}$ the closed subspace $\mbox{\rm
  Im}(I-\Pi_{\eps})=\mbox{\rm ker}(\Pi_{\eps})$, and by $T_{\eps}$ the
restriction of the operator $-\Delta_{\eps}$ to $\mathcal
K_{\eps}$. The operator $T_{\eps}$ is self-adjoint, with spectrum
$\sigma(T_{\eps})=\sigma(-\Delta_{\eps})\setminus
\{\lambda_{\eps}\}$. Furthermore, since $\lambda_j(\Omega\setminus
K_{\eps})\to \lambda_j(\Omega)$ for all $j\in\NN_1$ as $\eps\to 0^+$,
and since $\lambda_N(\Omega)$ is simple, there exists some $\delta>0$
such that $\mbox{\rm dist}(\lambda_{\eps},\sigma(T_{\eps}))\ge\delta$
for $\eps>0$ small enough. Using the spectral theorem for the operator
$T_{\eps}$, we get
\begin{equation*}
  \mbox{\rm dist}\left(\lambda_{\eps},\sigma(T_{\eps})\right)\left\|\widetilde{u}_{\eps}\right\|_{L^2(\Omega)}\le\left\|\left(T_{\eps}-\lambda_{\eps}\right)\widetilde{u}_{\eps}\right\|_{L^2(\Omega)},
\end{equation*} 
and therefore 
\begin{equation*}
  \left\|\psi_{\eps}-\Pi_{\eps}\psi_{\eps}\right\|_{L^2(\Omega)}\le \frac{\left\|\left(T_{\eps}-\lambda_{\eps}\right)\widetilde{u}_{\eps}\right\|_{L^2(\Omega)}}{\delta}=o\big(c_{\eps}^{1/2}\big) \quad\text{as
  }\eps\to0^+.
\end{equation*}
Consequently, we have
\begin{equation*}
  \left\|u_N-\Pi_{\eps}\psi_{\eps}\right\|_{L^2(\Omega)}\le \left\|V_{\eps}\right\|_{L^2(\Omega)}+\left\|\psi_{\eps}-\Pi_{\eps}\psi_{\eps}\right\|=o\big(c_{\eps}^{1/2}\big),
\end{equation*}
and therefore
\begin{equation*}
\left\|\Pi_{\eps}\psi_{\eps}\right\|_{L^2(\Omega)}=1+o\big(c_{\eps}^{1/2}\big) \quad\text{as
        }\eps\to0^+.
\end{equation*}
This implies in particular that $\Pi_{\eps}\psi_{\eps}$ is non-zero
for $\eps>0$ small enough, so that  we can define
\begin{equation*}
	u_{\eps}=\frac{\Pi_{\eps}\psi_{\eps}}{\left\|\Pi_{\eps}\psi_{\eps}\right\|_{L^2(\Omega)}},
\end{equation*}
an $L^2(\Omega)$-normalized eigenfunction of $-\Delta_{\eps}$
associated with $\lambda_{\eps}$. 
A simple computation shows that 
\begin{equation*}
	\|u_{\eps}-\psi_{\eps}\|_{L^2(\Omega)}=o\big(c_{\eps}^{1/2}\big)
\end{equation*}
and
\begin{equation*}
\|u_{\eps}-u_N\|_{L^2(\Omega)}=o\big(c_{\eps}^{1/2}\big)
\end{equation*}
as $\eps\to0^+$.
Taking the scalar product of equation \eqref{eq:Remainder} with $u_{\eps}$, we obtain 
\begin{align}\label{eq:Scalar}
  (\lambda_{\eps}-\lambda_N(\Omega))\langle u_{\eps} , \psi_{\eps}
  \rangle_{L^2(\Omega)} &= \lambda_N(\Omega)\langle
  u_N,V_{\eps}\rangle_{L^2(\Omega)}+\lambda_N(\Omega)\langle
  u_{\eps}-u_N,V_{\eps}\rangle_{L^2(\Omega)}\\
&=\lambda_N(\Omega)\langle
  u_N,V_{\eps}\rangle_{L^2(\Omega)}+o\left(c_{\eps}\right)\quad\text{as
  }\eps\to0^+.
\end{align}
On the other hand, since $\psi_{\eps}$ and $V_{\eps}$ are
$q$-orthogonal, we have
\begin{equation}\label{eq:CapScalar}
  c(\eps)=q(V_{\eps})=q(u_N-\psi_{\eps},V_{\eps})=q(u_N,V_{\eps})=\lambda_N(\Omega)\langle u_N,V_{\eps}\rangle_{L^2(\Omega)},
\end{equation}
using for the last equality the fact that $u_N$ is an eigenfunction of the quadratic form $q$, associated with the eigenvalue $\lambda_N(\Omega)$.
Reinjecting  \eqref{eq:CapScalar} into \eqref{eq:Scalar}, we finally obtain 
\begin{equation*}
  \lambda_{\eps}-\lambda_N(\Omega)=\frac{c_{\eps}+o(c_{\eps})}{\langle u_{\eps} , \psi_{\eps} \rangle}=c_{\eps}(1+o(1)),
\end{equation*}
as $\eps\to0^+$.
\end{proof}

\section{Continuity of the $u$-capacity}
In this appendix, we establish a
continuity result for the $u$-capacity with respect to concentration
  at zero capacity sets.

 \begin{prop}\label{p:conv_cap}
 If  $\{K_\eps\}_{\eps>0}$ is a family of
  compact sets contained in $\Omega\subset\RR^n$ concentrating to a compact set
  $K\subset\Omega$ with $\mbox{\rm Cap}_{\Omega}(K)=0$, then, for every
  $u\in H^1_0(\Omega)$, we have that 
$V_{K_\eps,u}\to V_{K,u}=0$ strongly in $H^1_0(\Omega)$ and
$\lim_{\eps\to 0^+}\mbox{\rm Cap}_{\Omega}(K_\eps,u)=\mbox{\rm Cap}_{\Omega}(K,u) =0$.
 \end{prop}
 \begin{proof}
   Testing equation \eqref{eq:4} for $V_{K_\eps,u}$ with
   $\varphi=V_{K_\eps,u}-u$ we obtain
 \begin{align}\label{eq:5}
\notag 0&=\int_{\Omega\setminus K_\eps}\nabla
V_{K_\eps,u}\cdot\nabla(V_{K_\eps,u}-u)\,dx\\
&=\int_{\Omega}\nabla
V_{K_\eps,u}\cdot\nabla(V_{K_\eps,u}-u)\,dx=\int_{\Omega}|\nabla
V_{K_\eps,u}|^2\,dx-\int_{\Omega}\nabla
V_{K_\eps,u}\cdot\nabla u\,dx.
 \end{align}
Since $V_{K_\eps,u}$ attains the minimum defining
$\mbox{\rm Cap}_{\Omega}(K_\eps,u)$, we have that $\int_\Omega|\nabla
V_{K_\eps,u}|^2\,dx\leq \int_\Omega|\nabla u|^2\,dx$, so that
$\{V_{K_\eps,u}\}_{\eps>0}$ is bounded in $H^1_0(\Omega)$. Hence,
along a sequence $\eps_k\to 0^+$, $V_{K_{\eps_k},u}\rightharpoonup V$
weakly in $H^1_0(\Omega)$ for some $V\in H^1_0(\Omega)$. Since
$\mbox{\rm Cap}_{\Omega}(K)=0$, we have that
$H^1_0(\Omega)=H^1_0(\Omega\setminus K)$ (see \cite[Proposition
2.1]{Courtois1995Holes}), hence $u-V\in H^1_0(\Omega\setminus
K)$. Moreover, for every $\varphi\in C^\infty_{\rm c}(\Omega\setminus
K)$, we have that $\varphi\in C^\infty_{\rm c}(\Omega\setminus
K_\eps)$ for $\eps$ sufficiently small, hence, passing to the limit in
\eqref{eq:4} for $V_{K_{\eps_k},u}$ as $k\to+\infty$, we obtain that   
\[
\int_{\Omega}\nabla
V\cdot\nabla\varphi\,dx=0.
\]
Hence $\int_{\Omega}\nabla V\cdot\nabla\varphi\,dx=0$ for every
$\varphi\in H^1_0(\Omega\setminus K)=H^1_0(\Omega)$. It follows that
$V=V_{K,u}=0$. Moreover, passing to the limit in  \eqref{eq:5}, we
obtain that 
\begin{equation*}
\lim_{k\to +\infty}\mbox{\rm Cap}_{\Omega}(K_{\eps_k},u)=
\lim_{k\to +\infty}\int_{\Omega}|\nabla
V_{K_{\eps_k},u}|^2\,dx=\lim_{k\to +\infty}\int_{\Omega} \nabla V_{K_{\eps_k},u}\cdot\nabla u\,dx=
\int_{\Omega} \nabla V\cdot\nabla u\,dx=0
\end{equation*}
We conclude that $\mbox{\rm Cap}_{\Omega}(K_{\eps_k},u)\to0$ and
$V_{K_{\eps_k},u}\to 0$ strongly in $H^1_0(\Omega)$ as $k\to
+\infty$. Since such limits do not depend on the subsequence, we reach
the conclusion.
 \end{proof}

 \paragraph{Acknowledgements} The authors are partially supported by
 the project ERC Advanced Grant 2013 n. 339958: ``Complex Patterns for
 Strongly Interacting Dynamical Systems -- COMPAT''. V. Felli is
 partially supported by PRIN-2012-grant ``Variational and perturbative
 aspects of nonlinear differential problems''.

\end{document}